\documentclass[12pt]{article}

\usepackage{amsmath}
\usepackage{amsfonts}
\usepackage{amsthm}
\usepackage{graphicx}
\numberwithin{equation}{section}
\newtheorem{theorem}{Theorem}

\renewcommand{\O}{\mathcal{O}}
\reversemarginpar

\begin{document}
\title{Riemann-Hilbert analysis for a Nikishin system}
\author{Guillermo L\'opez Lagomasino and Walter Van Assche \\ Universidad Carlos III de Madrid and KU Leuven}
\date{\today}
\maketitle

\begin{abstract}
In this paper we give the asymptotic behavior of type I multiple orthogonal polynomials for a Nikishin system of order two
with two disjoint intervals. We use the Riemann-Hilbert problem for multiple orthogonal polynomials and the steepest descent
analysis for oscillatory Riemann-Hilbert problems to obtain the asymptotic behavior in all relevant regions of the complex plane.
\end{abstract}

\section{Introduction}  \label{sec1}

It is well known \cite[Chap. 4]{NikiSor} that the polynomials appearing in Hermite-Pad\'e approximation satisfy a number of orthogonality relations,
and these polynomials are therefore known as multiple orthogonal polynomials (polyorthogonal polynomials, Hermite-Pad\'e polynomials).
Let $\vec{n} = (n_1,\ldots,n_r) \in \mathbb{Z}_+^r$ be a multi-index
and $|\vec{n}| = n_1+\cdots+n_r$.
Type I multiple orthogonal polynomials for measures $(\mu_1,\ldots,\mu_r)$ on the real line,
for which all the moments exist, 
are $(A_{\vec{n},1},\ldots,A_{\vec{n},r})$, where $\deg A_{\vec{n},j} \leq n_j-1$, for which
\[   \sum_{j=1}^r  \int A_{\vec{n},j}(x) x^k \, d\mu_j(x) = 0, \qquad 0 \leq k \leq |\vec{n}|-2. \]
The type II multiple orthogonal polynomial $P_{\vec{n}}$ is the polynomial of degree $\leq |\vec{n}|$ for which
\[   \int P_{\vec{n}}(x) x^k \, d\mu_j(x) = 0, \qquad 0 \leq k \leq n_j-1, \]
for all $j$ with $1 \leq j \leq r$. The corresponding Hermite-Pad\'e approximation for the functions
\[    f_j(z) = \int \frac{d\mu_j(x)}{z-x} , \qquad  1 \leq j \leq r, \]
for type I is that there exists a polynomial $B_{\vec{n}}$ such that
\[   \sum_{j=1}^r  A_{\vec{n},j}(z) f_j(z) - B_{\vec{n}}(z) = \O(z^{-|\vec{n}|}), \qquad z \to \infty, \]
and for type II Hermite-Pad\'e approximation there are $r$ polynomials $Q_{\vec{n},j}$ such that
\[    P_{\vec{n}}(z) f_j(z) - Q_{\vec{n},j}(z) = \O(z^{-n_j-1}), \qquad z \to \infty, \]
for $1 \leq j \leq r$. The existence of these multiple orthogonal polynomials is easy to justify, see below. However, their uniqueness, in general, is not guaranteed and one needs extra conditions on the system of measures $(\mu_1,\ldots,\mu_r)$, apart from the existence of all the moments. Two systems of measures for which all multi-indices have  unique solutions are Angelesco systems (the measures $\mu_j$ are supported on disjoint intervals)
and Nikishin systems (the measures $\mu_j$ are supported on the same interval, but their Radon-Nikodym derivatives can be described
in terms of a measure on a disjoint interval; see further for a more precise definition for $r=2$). 

Nikishin systems were introduced in 1980 by Nikishin \cite{Nikishin}, who claimed that multi-indices $\vec{n}=(n_1,n_2,\ldots,n_r) \in \mathbb{Z}_+^r \setminus \{\bf 0\}$ for which $n_1 \geq n_2 \geq \cdots \geq n_r$ are normal; that is, the corresponding multiple orthogonal polynomials exhibit maximum degree. Driver and Stahl \cite[p. 171]{DriverStahl} proved that all the multi-indices of a Nikishin system of order two are
normal, so that a Nikishin system of order 2 is perfect. Bustamante and L\'opez \cite{BusLop} had all the ingredients for such a proof but did not state
it or deduce it in their paper.
Recently Fidalgo Prieto and L\'opez Lagomasino \cite{FidLop} proved that every Nikishin system of order $r \geq 2$ is perfect.

We will be investigating multiple orthogonal polynomials 
for a Nikishin system of order two.
In particular, we will consider a Nikishin system of two positive measures $(\mu_1,\mu_2)$ on an interval $[a,b]$, for
which
\begin{equation}  \label{eq:w}
   d\mu_2(x) = w(x)\, d\mu_1(x), \quad w(x) =  \int_c^d \frac{d\sigma(t)}{x-t} ,
\end{equation}
where $\sigma$ is a positive measure on $[c,d]$ and the intervals $[a,b]$ and $[c,d]$ are disjoint.
We will assume (without loss of generality) that $c < d < a < b$, so that the interval $[c,d]$ is to the left of $[a,b]$ and hence the function $w$ in \eqref{eq:w} is positive on $[a,b]$. Furthermore, we assume that $\mu_1$ and $\sigma$ are absolutely continuous (with respect to Lebesgue measure), with
\begin{equation}  \label{eq:mu1}
    d\mu_1(x) = w_1(x) \,dx, \quad w_1(x) = (x-a)^\alpha (b-x)^\beta h_1(x), \qquad x \in [a,b],
\end{equation}
and
\begin{equation}   \label{eq:sigma}
   d\sigma(t) = w_2(t)\, dt , \quad  w_2(t) = (t-c)^\gamma (d-t)^\delta h_2(t), \qquad t \in [c,d],
\end{equation}
where $h_1$ is analytic in a neighborhood $\Omega_1$ of $[a,b]$ and $h_2$ is analytic in a neighborhood $\Omega_2$ of $[c,d]$, $h_1$ and $h_2$ have no
zeros at the endpoints of the intervals, and $\alpha,\beta,\gamma,\delta > -1$.

The asymptotic behavior of the ratio of two neighboring multiple orthogonal polynomials for Nikishin systems was investigated
earlier in \cite{AKLR}, \cite{ALR}, \cite{Abey}, \cite{FLLS}. In this paper, we wish to obtain strong asymptotics, i.e., asymptotics of the
individual polynomials, uniformly in the complex plane using the Riemann-Hilbert approach.
Using a different method, Aptekarev \cite{Apt} gave the strong asymptotic behavior of (type II) multiple orthogonal polynomials of a general Nikishin system ($r \geq 2$) for diagonal sequences $\vec{n} = (n,n,\ldots,n), n \in \mathbb{Z}_+.$

The Riemann-Hilbert problem for multiple orthogonal polynomials
was formulated in \cite{WVAGerKuijl}, and the authors gave the first few transformations of the Riemann-Hilbert problem
for Nikishin systems, but they did not perform the steepest descent analysis to get the full asymptotic behavior of the multiple orthogonal polynomials. Foulqui\'e Moreno \cite{Ana} showed how to set up the Riemann-Hilbert problem for a generalized Nikishin system (see \cite{GonRakhSor}),
but also did not work out the steepest descent analysis. For an Angelesco system the Riemann-Hilbert analysis was worked out in \cite{BranFidFM},
but their analysis is incomplete since they did not include the local analysis near the endpoints of the intervals (local parametrices).
The Riemann-Hilbert analysis for a system of measures (or Markov functions) generated by graphs was done in \cite{AptLysov}.
In fact, the diagonal case $m=n$ for type II multiple orthogonal polynomials is contained in \cite{AptLysov} and they used very much the same 
Riemann-Hilbert technique as we do in the present paper.
A full analysis of the Riemann-Hilbert problem for particular examples of multiple orthogonal polynomials is given in \cite{BleherKuijl} for 
multiple Hermite polynomials behaving like an Angelesco system, and in \cite{Lysov} for multiple Laguerre polynomials which 
behave like a Nikishin system. Two new phenomena in the steepest descent analysis of these Riemann-Hilbert problems were already demonstrated in 
\cite{AptBleherKuijl}, \cite{AptWVAYatt}, \cite{LysovWiel}: the global opening of the lenses and the transformation based on the generalized Nikishin equilibrium potentials. The Riemann-Hilbert analysis for ray sequences of indices, where $n/m \to \gamma$, was recently done in \cite{Yatt} (for an Angelesco system) and \cite{AptBogYatt} (for Frobenius-Pad\'e approximants). There is also high interest in the asymptotics of type I Nikishin systems with complex singular points, see \cite{RakhSuetin, KKPS, KPSC} and the references therein.

As mentioned before, there are two types of multiple orthogonal polynomials (and Hermite-Pad\'e approximants). In this paper, we will mainly
focus on type I multiple orthogonal polynomials, and the main result will be the asymptotic behavior
of the type I multiple orthogonal polynomials, which will be given in Section \ref{sec8}.
In Section \ref{sec9} we will work out the asymptotic behavior of the type II multiple orthogonal polynomials.

Since we are only dealing with $r=2$, we can simplify the notation.
A type I multiple orthogonal polynomials for the multi-index $(n,m)$ of the Nikishin system $(\mu_1,\mu_2)$ is given as a vector of two polynomials $(A_{n,m},B_{n,m})\neq (0,0)$, where $\deg A_{n,m} \leq n-1$ and $\deg B_{n,m} \leq m-1$, for which
\begin{equation}  \label{orthAB}
   \int_a^b  \Bigl( A_{n,m}(x) + w(x) B_{n,m}(x) \Bigr) x^k w_1(x)\, dx = 0, \qquad 0 \leq k \leq n+m-2,
\end{equation}
A type II multiple orthogonal polynomial $P_{n,m}$ for the multi-index $(n,m)$ is a polynomial of degree $\leq n+m$, not identically equal to zero,
for which
\[    \int_a^b x^k P_{n,m}(x) w_1(x)\, dx = 0, \qquad 0 \leq k \leq n-1, \]
\[    \int_a^b x^k P_{n,m}(x) w(x)w_1(x)\, dx = 0, \qquad 0 \leq k \leq m-1.  \]

The existence of $(A_{n,m},B_{n,m})$ and $P_{n,m}$ reduces (for each type) to solving a homogeneous system, on the coefficients of the polynomials, with one more equation than unknowns. So nontrivial solutions are guaranteed. Since $(\mu_1,\mu_2)$ is a perfect system, we know that for each $(n,m)$ any solution of one type or the other must verify  that $\deg A_{n,m} = n-1, \deg B_{n,m} = m-1,$ and $\deg P_{n,m} = n+m$. Other immediate consequences of perfectness is that
$(A_{n,m},B_{n,m})$ and $P_{n,m}$ are defined uniquely except for constant factors and
\[
 \int_a^b  \Bigl( A_{n,m}(x) + w(x) B_{n,m}(x) \Bigr) x^{n+m-1} w_1(x)\, dx = \kappa_{n,m} \neq 0,
\]
\[ \int_a^b x^n P_{n,m}(x) w_1(x)\, dx \neq 0, \qquad \int_a^b x^m P_{n,m}(x) w(x)w_1(x)\, dx \neq 0.\]
In the rest of the paper, we normalize $(A_{n,m},B_{n,m})$ so that $\kappa_{n,m} = 1$ and $P_{n,m}$ to be monic.

The perfectness of a Nikishin system of order 2 such as ours is a consequence of the following (extended) AT property (see \cite{Nikishin} for the original definition). For any pair of polynomials $(p,q) \neq (0,0), \deg p \leq n-1, \deg q \leq m-1$ with real coefficients and any $(n,m)$ the linear form $p + q w$
has at most $n+m-1$ zeros in $\mathbb{C} \setminus [c,d]$. For completeness we include a proof.

Let $n\geq m$ and assume that $p + q w$ has at least $n+m$ zeros in $\mathbb{C} \setminus [c,d]$. Since the coefficients of $(p,q)$ are real and $w$ is symmetric with respect to $\mathbb{R}$, the zeros of $p + q w$ come in conjugate pairs. Therefore, there exists a polynomial $W_{n,m}, \deg W_{n,m} \geq n+m,$ with real coefficients and zeros in $\mathbb{C}\setminus [c,d]$ such that $x^{k}(p + q w)/W_{n,m}, k=0,\ldots,m-1$ is holomorphic in $\mathbb{C}\setminus [c,d]$ and has a zero of order $\geq 2$ at infinity. Take a contour $\Gamma$ surrounding $[c,d]$ once in the positive direction
and separating it from $\infty$, with $[c,d]$ inside $\Gamma$ and $\infty$ and all the zeros of $W_{n,m}$ outside. Using Cauchy's theorem,  the definition of $w$, Fubini's theorem, and Cauchy's integral formula it follows that
\[0=\int_{\Gamma} \frac{z^k(p + qw)(z)\, dz}{W_{n,m}(z)}= \int_{\Gamma} \frac{z^k q(z)w(z)\, dz}{W_{n,m}(z)} =\int_c^d \frac{x^{k}q(x)}{W_{n,m}(x)}\, d\sigma(x), \quad k=0,\ldots,m-1.\]
Whence, $q$ has at least $m$ sign changes on $(c,d)$, but this is not possible since it has degree $\leq m-1$. Consequently, $q\equiv 0$ which implies that also $p\equiv 0$. Since $(p,q)\neq (0,0)$ we arrive at a contradiction.

The case when $n < m$ reduces to the previous one following the same scheme taking into consideration (see \cite[Lemma 6.3.5]{ST}) the well known fact that
\begin{equation}
\label{inverse}
\frac{1}{w(z)} = \ell(z) - \int_c^d \frac{d\tilde{\sigma}(x)}{z-x} = \ell(z) - \tilde{w}(z),
\end{equation}
where $\ell$ is a polynomial of degree $\leq 1$ and $\tilde{\sigma}$ is a finite positive measure on $[c,d]$. This transformation allows to view $(\mu_2,\mu_1)$ also as a Nikishin system. Incidentally, for general Nikishin systems (with $r\geq 2$), the proof of perfectness relies on the same basic ideas of the AT property for more general linear forms involving Nikishin systems and the reduction to the case when $n_1 \geq\cdots\geq n_r$, but now the reduction formulas turn out to be quite intricate. Unless otherwise stated, in the rest of the paper we will assume that $n \geq m$; however, taking account of \eqref{inverse}, the asymptotic formulas we obtain remain valid for sequences of multi-indices for which $n < m$ and appropriate conditions hold.

The Riemann-Hilbert problem for the type I multiple orthogonal polynomials is to find a matrix function
$X: \mathbb{C} \to \mathbb{C}^{3\times 3}$ such that
\begin{enumerate}
  \item $X$ is analytic in $\mathbb{C} \setminus [a,b]$.
  \item The boundary values $X_{\pm}(x) = \lim_{\epsilon \to 0+} X(x\pm i\epsilon)$ exist for $x \in (a,b)$ and satisfy
  \[   X_+(x) = X_-(x) \begin{pmatrix} 1 & 0 & 0 \\
                      -2\pi i w_1(x) & 1 & 0 \\
                      -2\pi i w(x)w_1(x) & 0 & 1 \end{pmatrix}, \qquad x \in (a,b).  \]
  \item Near infinity $X$ has the behavior
  \[   X(z) = \Bigl( \mathbb{I} + \O(1/z) \Bigr) \begin{pmatrix}  z^{-n-m} & 0 & 0 \\ 0 & z^n & 0 \\ 0 & 0 & z^m \end{pmatrix},
      \qquad z \to \infty. \]
  \item Near $a$ and $b$ the behavior is
  \[  X(z) = \begin{pmatrix} \O(r_a(z)) & \O(1) & \O(1) \\ \O(r_a(z)) & \O(1) & \O(1) \\ \O(r_a(z)) & \O(1) & \O(1) \end{pmatrix}, \qquad z \to a, \]
  \[  X(z) = \begin{pmatrix} \O(r_b(z)) & \O(1) & \O(1) \\ \O(r_b(z)) & \O(1) & \O(1) \\ \O(r_b(z)) & \O(1) & \O(1) \end{pmatrix}, \qquad z \to b, \]
  where
  \[    r_a(z) = \begin{cases}    |z-a|^\alpha, & -1 < \alpha < 0, \\
                                   \log |z-a|, & \alpha = 0, \\
                                   1,          &  \alpha > 0, \end{cases} \quad
        r_b(z) = \begin{cases}    |z-b|^\beta, & -1 < \beta < 0, \\
                                   \log |z-b|, & \beta = 0, \\
                                   1,          &  \beta > 0. \end{cases} \]
\end{enumerate}
The solution of this Riemann-Hilbert problem is
\[    X(z) = \begin{pmatrix}
             \displaystyle \int_a^b \frac{A_{n,m}(x)+w(x)B_{n,m}(x)}{z-x} w_1(x)\, dx &  A_{n,m}(z) &  B_{n,m}(z) \\
             \displaystyle c_1 \int_a^b \frac{A_{n+1,m}(x)+w(x)B_{n+1,m}(x)}{z-x}w_1(x)\, dx & c_1 A_{n+1,m}(z) & c_1 B_{n+1,m}(z) \\
             \displaystyle  c_2 \int_a^b \frac{A_{n,m+1}(x)+w(x)B_{n,m+1}(x)}{z-x}w_1(x)\, dx & c_2 A_{n,m+1}(z) & c_2 B_{n,m+1}(z)
             \end{pmatrix}, \]
where $c_1 =c_1(n,m)$ and $c_2=c_2(n,m)$ are such that
\[    c_1 A_{n+1,m}(z) = z^{n} + \textrm{lower order terms}, \qquad c_2 B_{n,m+1}(z) = z^{m} + \textrm{lower order terms}. \]
This Riemann-Hilbert problem was first formulated in \cite{WVAGerKuijl}, but we added the condition near the endpoints $a$ and $b$,
which is not needed when one has weights on the full real line. Such endpoint conditions were first introduced in \cite{KMVAV} for
orthogonal polynomials on $[-1,1]$.

\section{First transformation}  \label{sec2}
The weight function $w$ in the second measure $\mu_2$ is given in \eqref{eq:w} and it is a Stieltjes transform of a weight function on $[c,d]$,
so that $[c,d]$ is a branch cut for $w$. This is not yet visible in our Riemann-Hilbert problem. Our first transformation is intended
to bring these singularities into the Riemann-Hilbert problem and it was already suggested in \cite{WVAGerKuijl}.
We assume that $m \leq n$ and then the transformation is
\begin{equation}  \label{eq:U}
    U(z) = X(z) \begin{pmatrix} 1 & 0 & 0 \\
                                  0 & 1 & 0 \\
                                  0 & \displaystyle  \int_c^d \frac{d\sigma(t)}{z-t} & 1
                 \end{pmatrix}.
\end{equation}
Then $U: \mathbb{C} \to \mathbb{C}^{3\times 3}$ satisfies the following Riemann-Hilbert problem
\begin{enumerate}
  \item $U$ is analytic on $\mathbb{C} \setminus ([a,b] \cup [c,d])$.
  \item $U$ has jumps on $(a,b)$ and $(c,d)$ which are given by
  \[   U_+(x) = U_-(x) \begin{pmatrix} 1 & 0 & 0 \\ -2\pi i w_1(x) & 1 & 0 \\ 0 & 0 & 1 \end{pmatrix}, \qquad x \in (a,b), \]
  \[   U_+(x) = U_-(x) \begin{pmatrix} 1 & 0 & 0 \\ 0 & 1 & 0 \\ 0 &  -2\pi i w_2(x) & 1 \end{pmatrix}, \qquad x \in (c,d). \]
  \item Near infinity $U$ has the behavior (here we need $m \leq n$)
  \[    U(z) = \Bigl( \mathbb{I} + \O(1/z) \Bigr) \begin{pmatrix}  z^{-n-m} & 0 & 0 \\ 0 & z^n & 0 \\ 0 & 0 & z^m \end{pmatrix}, \qquad
             z \to \infty. \]
  \item Near $a$ and $b$ the behavior is
  \[  U(z) = \begin{pmatrix} \O(r_a(z)) & \O(1) & \O(1) \\ \O(r_a(z)) & \O(1) & \O(1) \\ \O(r_a(z)) & \O(1) & \O(1) \end{pmatrix}, \qquad z \to a, \]
  \[  U(z) = \begin{pmatrix} \O(r_b(z)) & \O(1) & \O(1) \\ \O(r_b(z)) & \O(1) & \O(1) \\ \O(r_b(z)) & \O(1) & \O(1) \end{pmatrix}, \qquad z \to b, \]
  and near $c$ and $d$
  \[  U(z) = \begin{pmatrix} \O(1) & \O(r_c(z)) & \O(1) \\ \O(1) & \O(r_c(z)) & \O(1) \\ \O(1) & \O(r_c(z)) & \O(1) \end{pmatrix}, \qquad z \to c, \]
  \[  U(z) = \begin{pmatrix} \O(1) & \O(r_d(z)) & \O(1) \\ \O(1) & \O(r_d(z)) & \O(1) \\ \O(1) & \O(r_d(z)) & \O(1) \end{pmatrix}, \qquad z \to d, \]
where
 \[    r_c(z) = \begin{cases}    |z-c|^\gamma, & -1 < \gamma < 0, \\
                                   \log |z-c|, & \gamma = 0, \\
                                   1,          &  \gamma > 0, \end{cases} \quad
        r_d(z) = \begin{cases}    |z-d|^\delta, & -1 < \delta < 0, \\
                                   \log |z-d|, & \delta = 0, \\
                                   1,          &  \delta > 0. \end{cases} \]
\end{enumerate}

\section{The vector equilibrium problem}  \label{sec3}

In this section, we assume that $n$ and $m=m(n)$ are related in such a way that
\begin{equation}
\label{limq}
  \lim_{n\to \infty} \frac{m}{n+m} = q_1, \qquad 0 < q_1 < 1.
\end{equation}
Since we will be working with the case when $m\leq n$, in fact $0 < q_1 \leq 1/2$. In the next section, it will be required that $q_1$ is a rational number.

The asymptotic distribution of the zeros of the type I (and type II) multiple orthogonal polynomials has been well studied and is given in terms of the solution of a vector equilibrium
problem for two probability measures $(\nu_1,\nu_2)$, where $\nu_1$ is supported on $[a,b]$ and $\nu_2$ is supported on $[c,d]$ or a subset
$[c^*,d]$ of $[c,d]$.
This was first worked out by Nikishin \cite{Nikishin86} and can be found in \cite[Chapter 5, \S7]{NikiSor}; for a more general
setting we refer to \cite{GonRakhSor} and \cite{FLLS}.
The support of $\nu_2$ can be a subset $[c^*,d]$ of $[c,d]$ whenever
$q_1 < \frac12$, see \cite[\S 5.6]{GonRakhSor}. In fact if $a,b,d,q_1$ are fixed, then there exists a $c^* < d$ such that $\textup{supp}(\nu_2)=[c^*,d]$
whenever $c < c^*$ and $\textup{supp}(\nu_2) = [c,d]$ whenever $c \geq c^*$.
Denote the logarithmic potential of a measure $\nu$ by
\[    U(z;\nu) = \int \log \frac{1}{|z-y|}\, d\nu(y). \]
When $\textup{supp}(\nu_2)=[c,d]$ the variational relations for the equilibrium problem are
\begin{align}
 2 U(x;\nu_1) - q_1 U(x;\nu_2) &= \ell_1, \qquad x \in [a,b], \label{eq:Uab} \\
 2q_1 U(x;\nu_2) - U(x;\nu_1)  &= \ell_2, \qquad x \in [c,d], \label{eq:Ucd}
\end{align}
where $\ell_1,\ell_2$ are constants, and when $\textup{supp}(\nu_2)=[c^*,d]$ with $c < c^*$, then
\begin{align}
 2 U(x;\nu_1) - q_1 U(x;\nu_2) &= \ell_1, \qquad x \in [a,b], \label{eq:Uab2} \\
 2q_1 U(x;\nu_2) - U(x;\nu_1)  &= \ell_2, \qquad x \in [c^*,d], \label{eq:Uc*d} \\
 2q_1 U(x;\nu_2) - U(x;\nu_1)  & > \ell_2, \qquad x \in [c,c^*).  \label{eq:Ucc*} 
\end{align}
It is well known that this equilibrium problem has a unique solution; for example, see \cite[Chapter 5, \S4]{NikiSor}.

The method goes as follows.  Notice that  the function $A_{n,m}(z) + w(z) B_{n,m}(z)$ has exactly $n+m-1$ simple zeros on $(a,b)$ at points $y_1,\ldots,y_{n+m-1}$, which depend on the
multi-index $(n,m)$ and no other zeros in $\mathbb{C}\setminus [c,d]$. Indeed, the AT property implies that this linear form can have in $\mathbb{C}\setminus [c,d]$ at most $n+m-1$ zeros whereas \eqref{orthAB} entails that it has at least $n+m-1$ sign changes on $(a,b)$. If we let $H_{n,m}$ be the monic polynomial of degree $n+m-1$ with simple zeros at these points, then from \eqref{orthAB}
\begin{equation}
\label{orto1}
 \int_a^b  H_{n,m}(x) \frac{ A_{n,m}(x) + w(x) B_{n,m}(x) }{H_{n,m}(x)} x^{k} w_1(x)\, dx = 0,\qquad k=0,\ldots,n+m-2.
\end{equation}
That is, $H_{n,m}$ is the monic orthogonal polynomial of degree $n+m-1$ on $[a,b]$ for the (varying) measure $\pm \frac{A_{n,m}(x)+w(x)B_{n,m}(x)}{H_{n,m}(x)}\, d\mu_1(x)$ (notice that ${(A_{n,m}+wB_{n,m})}/{H_{n,m}}$ has constant sign on $[a,b]$). The sign $\pm$ is such that the measure is positive on $[a,b]$.
If $n \geq m$, then $A_{n,m}(z) + w(z) B_{n,m}(z) = \O(z^{n-1})$ as $z \to \infty$, so that
$\Bigl(A_{n,m}(z) + w(z) B_{n,m}(z)\Bigr)/H_{n,m}(z) = \O(z^{-m})$ as $z \to \infty$. Hence, by Cauchy's theorem (for the exterior of $[c,d]$)
one has
\[  \frac{1}{2\pi i} \int_{\Gamma} \frac{A_{n,m}(z)+w(z)B_{n,m}(z)}{H_{n,m}(z)} z^k \, dz = 0, \qquad 0 \leq k \leq m-2, \]
whenever $\Gamma$ is a closed contour encircling $[c,d]$. If we take the contour in such a way that it stays away from $[a,b]$, then
\[  \frac{1}{2\pi i} \int_\Gamma \frac{A_{n,m}(z)}{H_{n,m}(z)} z^k \, dz = 0, \]
since the integrand is analytic on and inside $\Gamma$. Changing the order of integration and using \eqref{eq:w} then gives
\begin{equation}
\label{orto2}
   \int_c^d B_{n,m}(x) x^k  \frac{d\sigma(x)}{H_{n,m}(x)} = 0, \qquad 0 \leq k \leq m-2,
\end{equation}
so that $B_{n,m}$ is an orthogonal polynomial of degree $m-1$ on $[c,d]$ for the (varying) measure $(-1)^{n+m-1} d\sigma(x)/H_{n,m}(x)$, where the sign
makes the measure positive on $[c,d]$. Furthermore, one has
\begin{equation}   \label{AB/H}
   \frac{A_{n,m}(x)+w(x)B_{n,m}(x)}{H_{n,m}(x)} = \int_c^d \frac{B_{n,m}(t)}{x-t} \frac{d\sigma(t)}{H_{n,m}(t)}, \qquad x \notin [c,d],
\end{equation}
because we can write the left hand side using Cauchy's theorem as
\[  \frac{A_{n,m}(x)+w(x)B_{n,m}(x)}{H_{n,m}(x)} = -\frac{1}{2\pi i} \int_\Gamma \frac{A_{n,m}(z)+w(z)B_{n,m}(z)}{H_{n,m}(z)} \frac{dz}{z-x}, \]
where $\Gamma$ is a contour going counterclockwise around $[c,d]$ but not around $x$. Note that
\[       \frac{1}{2\pi i} \int_\Gamma \frac{A_{n,m}(z)}{H_{n,m}(z)} \frac{dz}{z-x} = 0, \]
since the integrand is analytic on and inside $\Gamma$. Consequently, using \eqref{eq:w} and interchanging the order of integration
\begin{eqnarray*}
   \frac{1}{2\pi i} \int_\Gamma \frac{w(z)B_{n,m}(z)}{H_{n,m}(z)} \frac{dz}{z-x}
  & = &\int_c^d d\sigma(t) \frac{1}{2\pi i} \int_\Gamma \frac{B_{n,m}(z)}{H_{n,m}(z)} \frac{dz}{(z-x)(z-t)} \\
  & = & - \int_c^d \frac{B_{n,m}(t)}{H_{n,m}(t)}  \frac{d\sigma(t)}{x-t},
\end{eqnarray*}
from which \eqref{AB/H} follows.

Due to \eqref{orto1}-\eqref{orto2}, the vector equilibrium problem corresponds to combining the equilibrium condition \eqref{eq:Uab}
for the asymptotic zero distribution $\nu_1$  of $H_{n,m}$ with external field
\[   \lim_{n,m \to \infty}  -\frac{1}{n+m} \log \frac{|A_{n,m}(x)+w(x)B_{n,m}(x)|}{|H_{n,m}(x)|}, \qquad x \in [a,b], \]
with the equilibrium condition \eqref{eq:Ucd} for the asymptotic distribution $\nu_2$ of the zeros of $B_{n,m}$ with external field
\[  \lim_{m \to \infty} \frac{1}{m} \log |H_{n,m}(x)|, \qquad x \in [c,d].  \]
Clearly, the external field on $[c,d]$ is $-U(x;\nu_1)/q_1$ (up to an additive constant). This gives the variational equation \eqref{eq:Ucd}
or \eqref{eq:Uc*d}--\eqref{eq:Ucc*}.
For the external field on $[a,b]$ we use
\eqref{AB/H} and the orthogonality of $B_{n,m}$ for the measure $d\sigma(t)/H_{n,m}(t)$ to find
\[      \frac{|A_{n,m}(x)+w(x)B_{n,m}(x)|}{|H_{n,m}(x)|} = \frac{1}{|B_{n,m}(x)|} \int_c^d \frac{B_{n,m}^2(t)}{x-t} \frac{d\sigma(t)}{|H_{n,m}(t)|}, \qquad x >d, \]
so that the external field on $[a,b]$ is (up to an additive constant)
\[   \lim_{n,m \to \infty} \frac{1}{n+m} \log |B_{n,m}(x)| = - q_1 U(x;\nu_2). \]
This gives the variational equation \eqref{eq:Uab} or \eqref{eq:Uab2}. 
For the case where $\textup{supp}(\nu_2) = [c^*,d]$, with $c < c^* <d$ one has the variational condition \eqref{eq:Uc*d} on $[c^*,d]$ which has to be supplemented with the inequality \eqref{eq:Ucc*} on $[c,c^*)$.
For details see \cite[Chapter 5, \S7]{NikiSor}.

\section{Normalizing the RHP}   \label{sec4}
The next transformation of the Riemann-Hilbert problem is to normalize it at infinity, but in such a way that we get nice jumps on the intervals.
This transformation will give the main term in the asymptotics outside the intervals.
For this we now introduce $g$-functions, which are the complex potentials of the measures $\nu_1$ and $\nu_2$:
\begin{equation}   \label{eq:g}
    g_1(x) = \int_a^b \log(x-t)\, d\nu_1(t), \quad g_2(x) = \int_c^d \log(x-t)\, d\nu_2(t),
\end{equation}
or, when $\textup{supp}(\nu_2)=[c^*,d]$,
\[     g_2(x) = \int_{c^*}^d \log(x-t)\, d\nu_2(t).   \]
For the logarithm we choose the branch cut on the negative real line.
Observe that for $x \in \mathbb{R}$
\begin{equation}   \label{eq:g1}
    g_1^\pm(x) = \begin{cases}
                    -U(x;\nu_1), & x > b, \\
                    -U(x;\nu_1) \pm i\pi, & x < a, \\
                    -U(x;\nu_1) \pm i\pi \varphi_1(x), & a \leq x \leq b,
                    \end{cases} \qquad
              \varphi_1(x) = \int_x^b d\nu_1(t) ,
\end{equation}
and similarly
\begin{equation}   \label{eq:g2}
    g_2^\pm(x) = \begin{cases}
                    -U(x;\nu_2), & x > d, \\
                    -U(x;\nu_2) \pm i\pi, & x < c \ (\textup{or } c^*), \\
                    -U(x;\nu_2) \pm i\pi \varphi_2(x), & c \ (\textup{or }c^*) \leq x \leq d,
                    \end{cases} \qquad
              \varphi_2(x) = \int_x^d d\nu_2(t) .
\end{equation}
We now introduce the second transformation
\begin{equation}  \label{eq:V}
      V(z) =  L  U(z)
               \begin{pmatrix}
                e^{(n+m)g_1(z)} & 0 & 0 \\ 0 & e^{-(n+m)g_1(z)+mg_2(z)} & 0 \\ 0 & 0 & e^{-mg_2(z)} \end{pmatrix}  L^{-1},
\end{equation}
where
\[    L = L(n,m) = \begin{pmatrix}  e^{-\frac{n+m}{3}(2\ell_1+\ell_2)} & 0 & 0 \\ 0 & e^{\frac{n+m}3(\ell_1-\ell_2)} & 0 \\
                 0 & 0 & e^{\frac{n+m}{3}(\ell_1+2\ell_2)}    \end{pmatrix}  . \]
This transformation indeed normalizes the behavior for $z \to \infty$:
\begin{equation}  \label{eq:V-infty}
    V(z) = \mathbb{I} + \O(1/z), \qquad z \to \infty ,
\end{equation}
which is a consequence of $g_i(z) = \log z + \O(1/z)$ for $i=1,2$ as $z \to \infty$.
The functions $e^{(n+m)g_1(z)}$ and $e^{mg_2(z)}$ have jumps on the intervals $(a,b)$ and $(c,d)$ or $(c^*,d)$ respectively, but are otherwise
analytic (the jumps over the branch cuts $(-\infty,a]$ and $(-\infty,c]$ or $(-\infty,c^*]$ disappear by taking the exponential),
hence $V$ is analytic on $\mathbb{C} \setminus ([a,b] \cup [c,d])$.
The price for normalizing the Riemann-Hilbert problem is that the jumps will be more complicated, but our choice of $g$-functions using the
equilibrium measures $(\nu_1,\nu_2)$ gives oscillatory jumps on the intervals.
Indeed, the jump over $(a,b)$
now is
\[   V_+(x) = V_-(x) \begin{pmatrix}
                      e^{(n+m)[g_1^+(x)-g_1^-(x)]} & 0 & 0 \\
                      -v_1(x) e^{(n+m)[g_1^+(x)+g_1^-(x)+\ell_1]-mg_2(x)} & e^{-(n+m)[g_1^+(x)-g_1^-(x)]} & 0 \\
                       0 & 0 & 1
                     \end{pmatrix} ,   \]
and over $(c,d)$ one has
\[   V_+(x) = V_-(x) \begin{pmatrix}
                      1 & 0 & 0 \\
                      0 & e^{m[g_2^+(x)-g_2^-(x)]} & 0 \\
                      0 & - v_2(x) e^{-(n+m)[g_1(x)-\ell_2] +m[g_2^+(x)+g_2^-(x)]} & e^{-m[g_2^+(x)-g_2^-(x)]}
                     \end{pmatrix}.  \]
Here we have used the functions
\[    v_1(x) = 2\pi i w_1(x) = 2\pi i(x-a)^\alpha(b-x)^\beta h_1(x), \quad  v_2(x) = 2\pi i w_2(x) = 2\pi i(x-c)^\gamma (d-x)^\delta h_2(x)  \]
to simplify the notation.
Recall from \eqref{eq:g1} that for $x \in (a,b)$
\begin{align*}    g_1^+(x)-g_1^-(x) &= 2\pi i \varphi_1(x),   \\
                  g_1^+(x)+g_2^-(x) &= -2U(x;\nu_1),
\end{align*}
and \eqref{eq:g2} gives for $x \in (c,d)$
\begin{align*}    g_2^+(x)-g_2^-(x) = 2\pi i \varphi_2(x),   \\
                  g_2^+(x)+g_2^-(x) = -2U(x;\nu_2),
\end{align*}
so that the jump for $V$ on $(a,b)$ is
\[   V_+(x) = V_-(x) \begin{pmatrix}
                      e^{2\pi i(n+m)\varphi_1(x)} & 0 & 0 \\
                      -v_1(x) e^{-2(n+m)U(x;\nu_1)+mU(x;\nu_2)+(n+m)\ell_1} & e^{-2\pi i(n+m) \varphi_1(x)} & 0 \\
                       0 & 0 & 1
                     \end{pmatrix} ,  \]
and on $(c,d)$ the jump is
\[   V_+(x) = V_-(x) \begin{pmatrix}
                      1 & 0 & 0 \\
                      0 & e^{2\pi im\varphi_2(x)} & 0 \\
                      0 & -v_2(x) e^{(n+m)U(x;\nu_1)-2mU(x;\nu_2)+(n+m)\ell_2} & e^{-2\pi im\varphi_2(x)}
                     \end{pmatrix}.  \]
If we take $q_1$ rational and $(n,m)$ so that
\begin{equation}
\label{rational}
 q_1=\frac{m}{n+m},
\end{equation}
and if $\textup{supp}(\nu_2)=[c,d]$ then the variational equations \eqref{eq:Uab}--\eqref{eq:Ucd} imply that the jumps
simplify to
\begin{equation}  \label{eq:Vab}
   V_+(x) = V_-(x) \begin{pmatrix}
                      e^{2\pi i(n+m)\varphi_1(x)} & 0 & 0 \\
                      -v_1(x)  & e^{-2\pi i(n+m) \varphi_1(x)} & 0 \\
                       0 & 0 & 1
                     \end{pmatrix} , \qquad x \in (a,b),
\end{equation}
and
\begin{equation}   \label{eq:Vcd}
   V_+(x) = V_-(x) \begin{pmatrix}
                      1 & 0 & 0 \\
                      0 & e^{2\pi im\varphi_2(x)} & 0 \\
                      0 & -v_2(x)  & e^{-2\pi im\varphi_2(x)}
                     \end{pmatrix}, \qquad x \in (c,d).
\end{equation}
If $\textup{supp}(\nu_2)=[c^*,d]$ with $c < c^* <d$, then \eqref{eq:Uc*d}--\eqref{eq:Ucc*} gives
\begin{equation}   \label{eq:Vc*d}
    V_+(x) = V_-(x) \begin{pmatrix}
                      1 & 0 & 0 \\
                      0 & e^{2\pi im\varphi_2(x)} & 0 \\
                      0 & -v_2(x)  & e^{-2\pi im\varphi_2(x)}
                     \end{pmatrix}, \qquad x \in (c^*,d).
\end{equation}
and
\begin{equation}   \label{eq:Vcc*}
V_+(x) = V_-(x) \begin{pmatrix}
                      1 & 0 & 0 \\
                      0 & 1 & 0 \\
                      0 & -v_2(x)e^{(n+m)\Phi(x)}  & 1
                     \end{pmatrix}, \qquad x \in [c,c^*),
\end{equation}
where $\Phi(x)=U(x;\nu_1)-2q_1U(x;\nu_2)+\ell_2 < 0$ on $[c,c^*)$.
Since $\varphi_1$ and $\varphi_2$ are real and positive functions, these jumps are oscillatory on the intervals $(a,b)$ and $(c,d)$ or $(c^*,d)$ respectively.
This is what we wanted to achieve, since it allows us to use the steepest descent method for oscillatory Riemann-Hilbert problems, which was
introduced by Deift and Zhou \cite{DeiftZhou1,DeiftZhou2}. Observe that the jumps \eqref{eq:Vab} and \eqref{eq:Vcd} are essentially reduced to
$2 \times 2$ jumps, which will make our work easier, because we can rely on some of the work done earlier (e.g., \cite{Deift}, \cite{KMVAV}).

\section{Opening the lenses}  \label{sec5}

A simple calculation shows that we can factorize the jump matrix in \eqref{eq:Vab} as
\[  \begin{pmatrix} (\Phi_1^+)^{n+m} & 0 & 0 \\ -v_1 & (\Phi_1^-)^{n+m} & 0 \\ 0 & 0 & 1 \end{pmatrix}
   = \begin{pmatrix}  1 & -(\Phi_1^+)^{n+m}/v_1 & 0 \\ 0 & 1 & 0 \\ 0 & 0 & 1 \end{pmatrix}
     \begin{pmatrix}  0 & 1/v_1 & 0 \\ -v_1 & 0 & 0 \\ 0 & 0 & 1 \end{pmatrix}
     \begin{pmatrix}  1 & -(\Phi_1^-)^{n+m}/v_1 & 0 \\ 0 & 1 & 0 \\ 0 & 0 & 1  \end{pmatrix}, \]
where we have used $\Phi_1^{\pm} = \exp(\pm 2\pi i \varphi_1)$. In a similar way, the jump matrix in \eqref{eq:Vcd} and \eqref{eq:Vc*d} factorizes as
\[  \begin{pmatrix} 1 & 0 & 0 \\ 0 & (\Phi_2^+)^m & 0 \\ 0 & -v_2 & (\Phi_2^-)^{m} \end{pmatrix}
   = \begin{pmatrix}  1 & 0 & 0 \\ 0 & 1 & -(\Phi_2^+)^m/v_2 \\ 0 & 0 & 1 \end{pmatrix}
     \begin{pmatrix}  1 & 0 & 0 \\ 0 & 0 & 1/v_2 \\ 0 & -v_2 & 0  \end{pmatrix}
     \begin{pmatrix}  1 & 0 & 0 \\ 0 & 1 & -(\Phi_2^-)^{m}/v_2  \\ 0 & 0 & 1  \end{pmatrix}, \]
where $\Phi_2^{\pm} = \exp(\pm 2\pi i\varphi_2)$. Therefore, instead of making one jump over $(a,b)$ with the jump matrix in \eqref{eq:Vab}
one can make three jumps over $(a,b)$, each with one of the matrices in the matrix factorization. The two outer matrices in the matrix
factorization contain oscillatory terms. By opening a lens with $[a,b]$ in the middle of the lens (see Figure \ref{fig:RHP-S}) one can make  jumps using
these outer matrices in the factorization over the lips of the lenses, but then one changes the Riemann-Hilbert problem inside the lens, where
only part of the jump in \eqref{eq:Vab} is done. The new Riemann-Hilbert matrix is
\[   S(z) = \begin{cases}
             V(z), & \textrm{outside the lens for $[a,b]$}, \\
             V(z)\begin{pmatrix} 1 & \Phi_1^{-(n+m)}/v_1 & 0 \\ 0 & 1 & 0  \\ 0 & 0 & 1 \end{pmatrix}, & \textrm{inside the lens, upper part}, \\
             V(z)\begin{pmatrix} 1 & -\Phi_1^{-(n+m)}/v_1 & 0 \\ 0 & 1 & 0 \\ 0 & 0 & 1 \end{pmatrix}, & \textrm{inside the lens, lower part}.
             \end{cases}  \]

\begin{figure}[ht]
\unitlength=2.2pt
\begin{picture}(200,50)(-20,30)
\thicklines
\put(20,50){\line(1,0){60}}
\put(100,50){\line(1,0){50}}
\put(20,50){\circle*{3}}
\put(18,43){$c$}
\put(80,50){\circle*{3}}
\put(80,43){$d$}
\put(100,50){\circle*{3}}
\put(98,43){$a$}
\put(150,50){\circle*{3}}
\put(150,43){$b$}
\put(100,50){\qbezier(0,0)(25,20)(50,0)}
\put(100,50){\qbezier(0,0)(25,-20)(50,0)}
\put(20,50){\qbezier(0,0)(30,20)(60,0)}
\put(20,50){\qbezier(0,0)(30,-20)(60,0)}
\put(125,50){\vector(1,0){3}}
\put(123,40){\vector(1,0){3}}
\put(123,60){\vector(1,0){3}}
\put(50,50){\vector(1,0){3}}
\put(49,40){\vector(1,0){1}}
\put(49,60){\vector(1,0){1}}
\put(135,60){$\Sigma_+^{a,b}$}
\put(135,36){$\Sigma_-^{a,b}$}
\put(30,60){$\Sigma_+^{c,d}$}
\put(30,36){$\Sigma_-^{c,d}$}
\end{picture}
\caption{Riemann-Hilbert problem for $S$}
\label{fig:RHP-S}
\end{figure}
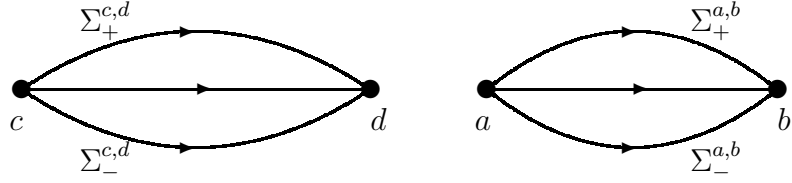
\bigskip

This can be done whenever the function $h_1$ (in the function $w_1$) is analytic in a region that contains the lens and if we can extend $\Phi_1$ analytically from $(a,b)$ to the region that contains the lens, in such a way that $\lim_{\epsilon \to 0+} \Phi_1(x\pm i\epsilon) = \Phi_1^\pm(x)$.
We also need to modify the Riemann-Hilbert problem near the lens for $[c,d]$: if $\textup{supp}(\nu_2)=[c,d]$ then
\[   S(z) = \begin{cases}
             V(z), & \textrm{outside the lens for $[c,d]$}, \\
             V(z)\begin{pmatrix} 1 & 0 & 0 \\ 0 & 1 & \Phi_2^{-m}/v_2  \\ 0 & 0 & 1 \end{pmatrix}, & \textrm{inside the lens, upper part}, \\
             V(z)\begin{pmatrix} 1 & 0 & 0 \\ 0 & 1 & -\Phi_2^{-m}/v_2 \\ 0 & 0 & 1 \end{pmatrix}, & \textrm{inside the lens, lower part},
             \end{cases}  \]
which can be done if $h_2$ (in the function $v_2$) is analytic in a region that contains the lens and if we can extend $\Phi_2$
analytically from $(c,d)$ to the region that contains the lens, in such a way that $\lim_{\epsilon \to 0+} \Phi_2(x\pm i\epsilon) = \Phi_2^\pm(x)$.
The jumps for the Riemann-Hilbert matrix over the six contours in Figure \ref{fig:RHP-S} are then given by
\[   S_+(z) =  \begin{cases}
    S_-(z) \begin{pmatrix}  1 & -\Phi_1^{-(n+m)}/v_1 & 0 \\ 0 & 1 & 0 \\ 0 & 0 & 1 \end{pmatrix}, & z \in \Sigma_+^{a,b} \cup \Sigma_-^{a,b}, \\
    S_-(z) \begin{pmatrix}  0 & 1/v_1 & 0 \\ -v_1 & 0 & 0 \\ 0 & 0 & 1  \end{pmatrix}, & z \in (a,b),
               \end{cases}  \]
and
\[   S_+(z) =  \begin{cases}
    S_-(z) \begin{pmatrix}  1 & 0 & 0 \\ 0 & 1 & -\Phi_2^{-m}/v_2  \\ 0 & 0 & 1 \end{pmatrix}, & z \in \Sigma_+^{c,d} \cup \Sigma_-^{c,d}, \\
    S_-(z) \begin{pmatrix}  1 & 0 & 0 \\ 0 & 0 & 1/v_2  \\ 0 & -v_2 & 0  \end{pmatrix}, & z \in (c,d).
               \end{cases}  \]
If $\textup{supp}(\nu_2)=[c^*,d]$ with $c < c^*<d$, then we open the lens only over $(c^*,d)$, see Figure \ref{fig:RHP-S*}, and
\[   S_+(z) =  \begin{cases}
    S_-(z) \begin{pmatrix}  1 & 0 & 0 \\ 0 & 1 & -\Phi_2^{-m}/v_2  \\ 0 & 0 & 1 \end{pmatrix}, & z \in \Sigma_+^{c^*,d} \cup \Sigma_-^{c^*,d}, \\
    S_-(z) \begin{pmatrix}  1 & 0 & 0 \\ 0 & 0 & 1/v_2  \\ 0 & -v_2 & 0  \end{pmatrix}, & z \in (c^*,d), \\
    S_-(z) \begin{pmatrix}  1 & 0 & 0 \\ 0 & 1 & 0 \\ 0 & -v_2 e^{(n+m)\Phi} & 1 \end{pmatrix}, & z \in [c,c^*).
               \end{cases}  \]

\begin{figure}[ht]
\unitlength=2.2pt
\begin{picture}(200,50)(-20,30)
\thicklines
\put(20,50){\line(1,0){60}}
\put(100,50){\line(1,0){50}}
\put(20,50){\circle*{3}}
\put(18,43){$c$}
\put(30,50){\circle*{3}}
\put(28,43){$c^*$}
\put(80,50){\circle*{3}}
\put(80,43){$d$}
\put(100,50){\circle*{3}}
\put(98,43){$a$}
\put(150,50){\circle*{3}}
\put(150,43){$b$}
\put(100,50){\qbezier(0,0)(25,20)(50,0)}
\put(100,50){\qbezier(0,0)(25,-20)(50,0)}
\put(30,50){\qbezier(0,0)(25,20)(50,0)}
\put(30,50){\qbezier(0,0)(25,-20)(50,0)}
\put(125,50){\vector(1,0){3}}
\put(123,40){\vector(1,0){3}}
\put(123,60){\vector(1,0){3}}
\put(26,50){\vector(1,0){1}}
\put(55,50){\vector(1,0){3}}
\put(56,40){\vector(1,0){1}}
\put(56,60){\vector(1,0){1}}
\put(135,60){$\Sigma_+^{a,b}$}
\put(135,36){$\Sigma_-^{a,b}$}
\put(35,60){$\Sigma_+^{c^*,d}$}
\put(35,36){$\Sigma_-^{c^*,d}$}
\end{picture}
\caption{Riemann-Hilbert problem for $S$ when $\textup{supp}(\nu_2)=[c^*,d]$}
\label{fig:RHP-S*}
\end{figure}
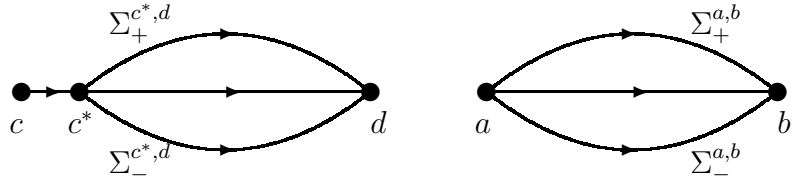
\bigskip

We already know that $\Phi < 0$ on $[c,c^*)$. 
We claim that $|\Phi_1|>1$ on the lips  $\Sigma_\pm^{a,b}$ of the lens for $[a,b]$ (except at $a$ and $b$) and that $|\Phi_2| >1$ on the lips
$\Sigma_\pm^{c,d}$ or $\Sigma_\pm^{c^*,d}$ of the lens for $[c,d]$ or $[c^*,d]$ (except at the endpoints), provided that these lenses are thin enough.
The function $\Phi_1^+$, which is defined on $[a,b]$ by $\Phi_1^+(x) = \exp(2\pi i \varphi_1(x))$, with $\varphi_1$ given by \eqref{eq:g1},
can be extended to the function $\Phi_1(z) = \exp( 2\pi i \varphi_1(z))$ with $\Im z >0$, where $\varphi_1$ is given by
\[    \varphi_1(z) = \int_z^b \nu_1'(t)\, dt,  \]
where $\nu_1'(t) = m_1(t)(t-a)^{-1/2}(b-t)^{-1/2}$ is the density of the measure $\nu_1$, which is absolutely continuous with square root singularities at the endpoints, and with $m_1$ a positive analytic function on $[a,b]$. Write $\varphi_1(z) = u(x,y) + i v(x,y)$, with $z=x+iy$, then on the interval $[a,b]$
we have $u(x,0)= \int_x^b \nu_1'(t)\, dt$ and $v(x,0) = 0$, so that $\frac{\partial u}{\partial x} = -\nu_1'(x) < 0$ for $x \in (a,b)$. The Cauchy-Riemann
equations then imply that $\frac{\partial v}{\partial y} <0$, so that $v(x,y)$ is a decreasing function of $y$, and hence $v(x,y) <0$ for $x \in (a,b)$
when $y >0$ and close to $0$. Then $|\Phi_1(z)| = \exp(- 2\pi v(x,y)) > 0$ for $y>0$ and close to zero and hence $|\Phi_1| >1$ on $\Sigma_+^{a,b}$,
away from $a$ and $b$.
In a similar way we have that $\Phi_1^-(x) = \exp(-2\pi i \varphi_1(x))$ on $[a,b]$, so in the lower complex plane $|\Phi_1(z)| = \exp(2\pi v(x,y))$, with
$y < 0$ and close to zero. Since $v(x,y)$ is a decreasing function of $y$ and $v(x,0)=0$, we have that $v(x,y) >0$ for $y<0$ and small enough,
meaning that $|\Phi_1| > 1$ on $\Sigma_-^{a,b}$, away from the points $a$ and $b$.
The same reasoning can be given to show that $|\Phi_2| > 1$ on the lips $\Sigma_{\pm}^{c,d}$ or $\Sigma_{\pm}^{c^*,d}$ away from the points $c$ or $c^*$ and $d$. 

\section{The global parametrix}  \label{sec6}

The global parametrix is the solution of the Riemann-Hilbert problem for $S$ if we ignore the jumps on the lips of the lenses, which for $n,m \to \infty$
converge to the identity matrix. So we look for a $3\times 3$ matrix $N$ which is analytic in $\mathbb{C} \setminus ([a,b] \cup [c,d])$ with jumps
\[    N_+(x) = N_-(x) \begin{pmatrix} 0 & 1/v_1 & 0 \\ -v_1 & 0 & 0 \\ 0 & 0 & 1  \end{pmatrix}, \qquad x \in (a,b), \]
\[    N_+(x) = N_-(x) \begin{pmatrix} 1 & 0 & 0 \\ 0 & 0 & 1/v_2 \\ 0 & -v_2 & 0  \end{pmatrix}, \qquad x \in (c,d). \]
The case $\textup{supp}(\nu_2)=[c^*,d]$ is similar and one only needs to change $c$ to $c^*$.

\begin{figure}[ht]
\begin{picture}(200,210)(0,250)
\unitlength=3.5pt
\put(0,100){\line(1,0){100}}
\put(0,100){\line(2,1){20}}
\put(20,110){\line(1,0){100}}
\put(100,100){\line(2,1){20}}
\put(0,80){\line(1,0){100}}
\put(0,80){\line(2,1){20}}
\put(20,90){\line(1,0){100}}
\put(100,80){\line(2,1){20}}
\put(0,120){\line(1,0){100}}
\put(0,120){\line(2,1){20}}
\put(20,130){\line(1,0){100}}
\put(100,120){\line(2,1){20}}
\multiput(60,105)(0,2){10}{\line(0,1){1}}
\multiput(90,105)(0,2){10}{\line(0,1){1}}
\multiput(20,105)(0,-2){10}{\line(0,-1){1}}
\multiput(55,105)(0,-2){10}{\line(0,-1){1}}
\thicklines
\put(60,105){\line(1,0){30}}
\put(60,125){\line(1,0){30}}
\put(20,105){\line(1,0){35}}
\put(20,85){\line(1,0){35}}
\put(60,105){\circle*{1}}
\put(90,105){\circle*{1}}
\put(20,105){\circle*{1}}
\put(55,105){\circle*{1}}
\put(60,125){\circle*{1}}
\put(90,125){\circle*{1}}
\put(20,85){\circle*{1}}
\put(55,85){\circle*{1}}
\put(100,125){$\mathfrak{R}_0$}
\put(100,105){$\mathfrak{R}_1$}
\put(100,85){$\mathfrak{R}_2$}
\put(60,126.5){$a$}
\put(90,126.5){$b$}
\put(20,82){$c$}
\put(55,82){$d$}
\end{picture}
\caption{The Riemann surface $\mathfrak{R}$}
\label{fig:R}
\end{figure}
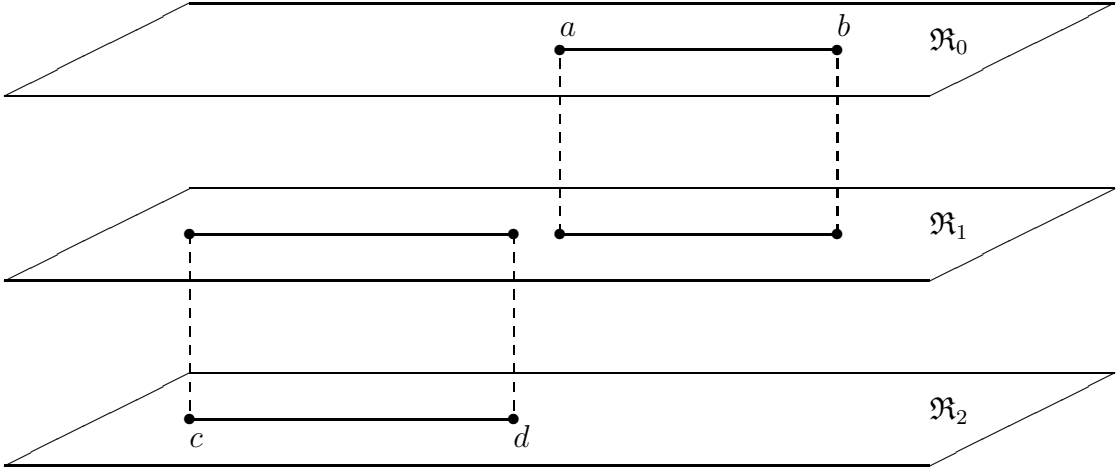

We solve this in two steps, as was done, e.g., in \cite{BranFidFM} or \cite{KlaasArno} for an Angelesco system. First, we need the Szeg\H{o} functions for $(v_1,v_2)$ for the geometry of the Riemann surface $\mathfrak{R}$
with three sheets $\mathfrak{R}_0,\mathfrak{R}_1,\mathfrak{R}_2$,
which is of genus 0 and has branch points $a,b,c,d$. The interval $[a,b]$ on the first sheet $\mathfrak{R}_0$ is connected in the usual
way to the interval $[a,b]$ on $\mathfrak{R_1}$, and $[c,d]$ on the third sheet $\mathfrak{R}_2$ is connected to $[c,d]$ on the second sheet
$\mathfrak{R}_1$ (see Figure \ref{fig:R}).
These Szeg\H{o} functions are analytic functions $D_0,D_1,D_2$ on $\overline{\mathbb{C}} \setminus ([a,b] \cup [c,d])$ which satisfy the boundary conditions
\begin{equation} \label{eq:6.1}
      \begin{cases}
         D_1^+(x) = v_1(x) D_0^-(x), \\
         D_1^-(x) = v_1(x) D_0^+(x), \\
         D_2^+(x) = D_2^-(x),
       \end{cases}   \qquad x \in [a,b],
\end{equation}
and
\begin{equation} \label{eq:6.2}
   \begin{cases}
         D_0^+(x) = D_0^-(x), \\
         D_2^+(x) = v_2(x) D_1^-(x) , \\
         D_2^-(x) = v_2(x) D_1^+(x),
      \end{cases} \qquad x \in [c,d],
\end{equation}
and for which the limit for $z \to \infty$ does not vanish: $D_0(\infty) \neq 0$, $D_1(\infty) \neq 0,$ and  $D_2(\infty) \neq 0$.
\bigskip

Then, with these functions we can define the matrix
\begin{equation}   \label{eq:N}
   N_0(z) = \begin{pmatrix} D_0(\infty) & 0 & 0 \\ 0 & D_1(\infty) & 0 \\ 0 & 0 & D_2(\infty) \end{pmatrix}^{-1}
              N(z) \begin{pmatrix}  D_0(z) & 0 & 0 \\ 0 & D_1(z) & 0 \\ 0 & 0 & D_2(z) \end{pmatrix},
\end{equation}
and this has the same behavior as $N$ when $z\to \infty$, but has a simpler jump on the intervals
\begin{equation}
\label{jump1}
     N_0^+(x) = N_0^-(x) \begin{pmatrix} 0 & 1 & 0 \\ -1 & 0 & 0 \\ 0 & 0 & 1  \end{pmatrix}, \qquad x \in (a,b),
\end{equation}
\begin{equation}
\label{jump2} N_0^+(x) = N_0^-(x) \begin{pmatrix} 1 & 0 & 0 \\ 0 & 0 & 1 \\ 0 & -1 & 0  \end{pmatrix}, \qquad x \in (c,d).
\end{equation}
\bigskip
The second step consists in finding an expression for $N_0$.

For both steps, we need a conformal mapping $\psi: \mathfrak{R} \to \overline{\mathbb{C}}$, $\psi_j(x)=y$ $(j=0,1,2)$ between the 
Riemann surface $\mathfrak{R}$ (see Figure \ref{fig:R}) and the extended complex plane $\overline{\mathbb{C}}$ (see Figure \ref{fig:Rimage}).
\begin{figure}[ht]
\unitlength=0.7mm
\centering
\begin{picture}(220,130)(0,-20)
\put(0,50){\line(1,0){220}}
\put(10,51){\line(0,-1){2}}
\put(60,51){\line(0,-1){2}}
\put(160,51){\line(0,-1){2}}
\put(210,51){\line(0,-1){2}}
\put(9,44){$-2$}
\put(59,44){$-1$}
\put(112,44){$0$}
\put(160,45){$1$}
\put(211,44){$2$}
\put(110,110){\line(0,-1){120}}
\put(109,100){\line(1,0){2}}
\put(109,75){\line(1,0){2}}
\put(109,25){\line(1,0){2}}
\put(109,0){\line(1,0){2}}
\put(112,100){$1$}
\put(112,75){$0.5$}
\put(112,25){$-0.5$}
\put(112,0){$-1$}
\put(60,50){\qbezier(23.63,0)(23.63,9.79)(16.71,16.71)}
\put(60,50){\qbezier(0,23.63)(9.79,23.63)(16.71,16.71)}
\put(60,50){\qbezier(0,23.63)(-9.79,23.63)(-16.71,16.71)}
\put(60,50){\qbezier(-23.63,0)(-23.63,9.79)(-16.71,16.71)}
\put(60,50){\qbezier(23.63,0)(23.63,-9.79)(16.71,-16.71)}
\put(60,50){\qbezier(0,-23.63)(9.79,-23.63)(16.71,-16.71)}
\put(60,50){\qbezier(0,-23.63)(-9.79,-23.63)(-16.71,-16.71)}
\put(60,50){\qbezier(-23.63,0)(-23.63,-9.79)(-16.71,-16.71)}
\put(130,90){$\psi(\mathfrak{R}_1)$}
\put(50,60){$\psi(\mathfrak{R}_2)$}
\put(170,60){$\psi(\mathfrak{R}_0)$}
\put(170,60){\vector(-1,-1){8}}
\put(160,50){\qbezier(6.26,0)(6.26,2.59)(4.43,4.43)}
\put(160,50){\qbezier(0,6.26)(2.59,6.26)(4.43,4.43)}
\put(160,50){\qbezier(0,6.26)(-2.59,6.26)(-4.43,4.43)}
\put(160,50){\qbezier(-6.26,0)(-6.26,2.59)(-4.43,4.43)}
\put(160,50){\qbezier(6.26,0)(6.26,-2.59)(4.43,-4.43)}
\put(160,50){\qbezier(0,-6.26)(2.59,-6.26)(4.43,-4.43)}
\put(160,50){\qbezier(0,-6.26)(-2.59,-6.26)(-4.43,-4.43)}
\put(160,50){\qbezier(-6.26,0)(-6.26,-2.59)(-4.43,-4.43)}
\thicklines
\put(61,73.63){\vector(1,0){0}}
\put(162,56.26){\vector(1,0){0}}
\put(44,74){$\Gamma_2^+$}
\put(44,24){$\Gamma_2^-$}
\put(150,57){$\Gamma_1^-$}
\put(150,40){$\Gamma_1^+$}
\end{picture}
\caption{The image of $\mathfrak{R}$ through the mapping $\psi$ for the intervals $[-5,-1]$ and $[1,2]$}
\label{fig:Rimage}
\end{figure}
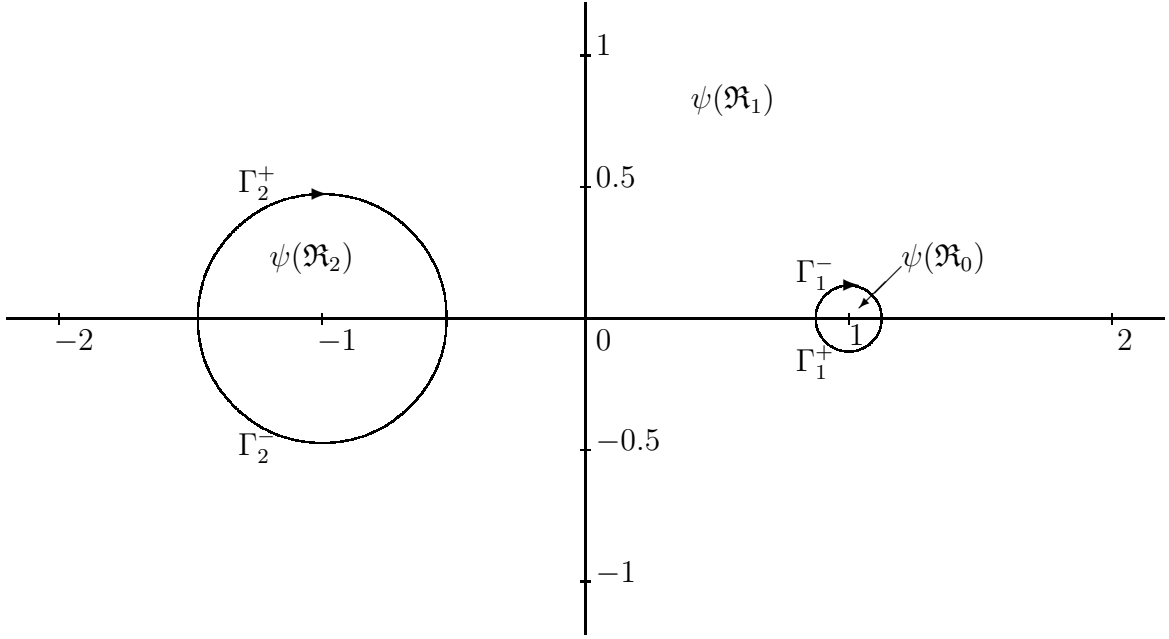

A way to obtain such a mapping was described in \cite{LPRY}. Using an affine transformation, if necessary, without loss of generality we can assume that $[a,b] = [1,\lambda]$ and $[c,d] = [-\mu,-1], \lambda, \mu > 0$. In \cite[Theorem 3.1]{LPRY} it was proved that the rational function $H(y)$ below establishes a one to one correspondence between $\overline{\mathbb{C}}$ and $\mathfrak{R}$. Here,
\begin{equation}
\label{rational}
H(y) = x = h+y+ \frac{Ay}{1-y} + \frac{By}{1+y}
\end{equation}
with constants $A,B$ and $h$ given by
\begin{align*}
     A &= \frac14 (1-\hat\beta)(1-\hat\alpha)(1-\hat{a})(1-\hat{b}), \\
     B &= \frac14 (1+\hat\beta)(1+\hat\alpha)(1+\hat{a})(1+\hat{b}), \\
     h &= \frac14 (\hat{a}+\alpha) \left( 2\hat{a}\alpha - \frac{(\hat{a}-\alpha)^2}{1-\hat{a}\alpha} \right),
\end{align*}
where $\hat{\beta}, \hat{\alpha}, \hat{a}, \hat{b}\,\, (\hat\beta < -1 < \hat\alpha < \hat{a} < 1 < \hat{b})$ are the critical points of $H$,
which are univocally determined as solutions of some algebraic equations depending solely on $\mu$ and $\lambda$.
Specifically, $\hat\beta$ and $\hat{b}$ are the solutions of the quadratic equation
\[ x^2 + (\hat{a} + \hat\alpha) x + \frac{(\hat{a}-\hat\alpha)^2}{1- \hat{a}\hat\alpha} - 3 = 0,\]
whereas $\hat\alpha$ and $\hat{a}$ are the unique solutions of the algebraic system
\begin{align*}
 (\mu - \lambda)(\hat{a}-\hat\alpha)^3 & = 2(\hat{a} +\hat\alpha)[(9-\hat{a}\hat\alpha)(1-\hat{a}\hat\alpha) - (\hat{a}-\hat\alpha)^2]   \\
 (\mu + \lambda)^2(\hat{a}-\hat\alpha)^6 & =4 (3+\hat{a}\hat\alpha)^2(1-\hat{a}\hat\alpha)[((\hat{a}+\hat\alpha)^2 +12)(1-\hat{a}\hat\alpha) - 4(\hat{a}-\hat\alpha)^2].
\end{align*}
Note that we used the notation $\hat{a}, \hat\alpha, \hat{b}$ and $\hat{\beta}$ because $a$ and $b$ are already used for the endpoints of the interval $[a,b]$ and $\alpha, \beta$ for the exponents in $w_1$ (see \eqref{eq:mu1}).
The $\hat{a}, \hat{b},\hat\alpha, \hat\beta,$ correspond to $a,b,\alpha,\beta$ in \cite{LPRY}, respectively. (When $\mu = \lambda$, that is if the intervals $[a,b], [c,d]$ have equal length, these equations reduce substantially and can be solved exactly with radicals, see \cite{LPRY}).

In other words, we can take $\psi$ as the solution of the cubic equation
\[    y^3-(x+A-B-h)y^2 - (1+A+B)y +x-h = 0, \]
Let   $\psi_0,\psi_1,\psi_2$ be the branches of $\psi$ corresponding to $\mathfrak{R}_0,\mathfrak{R}_1, \mathfrak{R}_2$, respectively. Denote
$\widetilde{\mathfrak{R}}_j = \psi(\mathfrak{R}_j), j=0,1,2$ (see Figure 3). We have $\psi(\infty^{(0)}) = \psi_0(\infty) = 1, \psi(\infty^{(1)}) = \psi_1(\infty) =\infty, \psi(\infty^{(2)}) = \psi_2(\infty) = -1$.
Let us orient the closed curves $\partial \widetilde{\mathfrak{R}}_0$ and $\partial \widetilde{\mathfrak{R}}_2$ so that the left ($+$) sides induced by the orientation coincides with the regions $\widetilde{\mathfrak{R}}_0$ and $\widetilde{\mathfrak{R}}_2$.

Let us find $D = \mbox{diag}(D_0,D_1,D_2)$ holomorphic in $\overline{\mathbb{C}} \setminus ([-\mu,-1] \cup [1,\lambda])$ verifying \eqref{eq:6.1}-\eqref{eq:6.2}. We seek the functions $D_j$ in the form $D_j(y) = \widehat{D}(\psi(x))$, with $\psi_j(x)=y$, $j=0,1,2$, 
for some function $\widehat{D}$ verifying:
\begin{itemize}
\item[i)] $\widehat{D} \in H(\overline{\mathbb{C}} \setminus(\partial\widetilde{\mathfrak{R}}_0 \cup \partial\widetilde{\mathfrak{R}}_2))$ and nowhere zero.
\item[ii)]$\widehat{D}_-(x) = v_1(H(x)) \widehat{D}_+(x), x \in \partial \widetilde{\mathfrak{R}}_0.$
\item[iii)] $\widehat{D}_+(x) = v_2(H(x)) \widehat{D}_-(x), x \in \partial \widetilde{\mathfrak{R}}_2.$
\end{itemize}
where $H(y)$ is defined in \eqref{rational}. This is consistent with the orientation and \eqref{eq:6.1}-\eqref{eq:6.2}.

Applying a branch of the logarithm to $\mbox{\rm ii)-iii)}$ we see that finding $\widehat{D}$ reduces to solving a scalar additive Riemann-Hilbert problem with boundary conditions
\begin{itemize}
\item $\log \widehat{D}_-(x) = \log v_1(H(x)) + \log \widehat{D}_+(x), \quad x \in \partial \widetilde{\mathfrak{R}}_0,$
\item $\log \widehat{D}_+(x) = \log v_2(H(x)) + \log \widehat{D}_-(x), \quad x \in \partial \widetilde{\mathfrak{R}}_2.$
\end{itemize}
Using the Sokhotsky-Plemelj formula, we find that
\[\widehat{D}(w) = C \exp \left(\frac{1}{2\pi i}\int_{\partial \widetilde{\mathfrak{R}}_2}  \frac{\log v_2(H(x))}{x - w} dx- \frac{1}{2\pi i}\int_{\partial \widetilde{\mathfrak{R}}_0}\frac{\log v_1(H(x))}{x - w} dx\right),\]
where $C$ is an arbitrary constant and the integration is performed according to the orientation selected.
Hence, defining for $j=0,1,2$
\begin{equation}   \label{eq:D}
 {D}_j(z) = C \exp \left(\frac{1}{2\pi i}\int_{\partial \widetilde{\mathfrak{R}}_2}  \frac{\log v_2(H(x))}{x - \psi_j(z)} dx- \frac{1}{2\pi i}\int_{\partial \widetilde{\mathfrak{R}}_0}\frac{\log v_1(H(x))}{x - \psi_j(z)} dx\right),
\end{equation}
we obtain a solution of \eqref{eq:6.1}--\eqref{eq:6.2}. Moreover, notice that $D_0D_1D_2$ can be extended to an 
entire function on $\overline{\mathbb{C}}$. Indeed, the boundary conditions imply that it is analytic except possibly at $\{-\mu,-1,1,\lambda\}$. At these points the conditions on $v_j, j=1,2$ imply that there it behaves like $\mathcal{O}(1)$. It is also bounded at infinity; therefore, it is constant. We can take $C$ so that $D_0D_1D_2\equiv 1$.

Now we find $N_0$ in the form
\begin{equation}
\label{eq:N0}
N_0(z) =
\left(
\begin{array}{ccc}
N_{1}(\psi_0(z)) & N_{1}(\psi_1(z)) & N_{1}(\psi_2(z)) \\
N_{2}(\psi_0(z)) & N_{2}(\psi_1(z)) & N_{2}(\psi_2(z)) \\
N_{2}(\psi_0(z)) & N_{3}(\psi_1(z)) & N_{3}(\psi_2(z)) \\
\end{array}
\right) ,
\end{equation}
where $N_1,N_2,N_3$ are appropriate algebraic functions defined on $\overline{\mathbb{C}}$. Set
\[\Gamma_1^+ = \psi_{0,+}([1,\lambda]), \quad \Gamma_1^- = \psi_{0,-}([1,\lambda]), \quad \Gamma_2^+ = \psi_{1,+}([-\mu,-1]), \quad \Gamma_2^- = \psi_{1,-}([-\mu,-1]). \]
We have $\partial \widetilde{\mathfrak{R}}_0  = \Gamma_1^+ \cup \Gamma_1^-$ and $\partial \widetilde{\mathfrak{R}}_2  = \Gamma_2^+ \cup \Gamma_2^-$. It is not difficult to verify that $\Gamma_1^-, \Gamma_2^+$ lie in the upper half plane whereas $\Gamma_1^+, \Gamma_2^-$ in the lower half plane,
see Figure \ref{fig:Rimage}.
The boundary conditions \eqref{jump1}-\eqref{jump2} imply that
\[N_{j,+}(x) = - N_{j,-}(x), \,\,\, x\in \Gamma_1^+ \cup \Gamma_2^+,  \quad N_{j,+}(x) = N_{j,-}(x), \,\,\, x \in \Gamma_1^- \cup \Gamma_2^-,  \quad j=1,2,3.\]
Recall that $\hat\beta,\hat\alpha,\hat{a},\hat{b}$ are the critical points of $H$ which has simple poles at $-1,1,\infty$. The condition which the matrix function $N_0$ should verify at infinity reduces to requiring that $N_j(\tau_i) = \delta_{i,j}$ where $\tau_i$ stands for $-1,1,$ or $\infty$.
Define
\[r(z) = \sqrt{(z-\hat\beta)(z-\hat\alpha)(z-\hat{a})(z-\hat{b})}\]
with a branch cut along $\Gamma_1^+ \cup \Gamma_2^+$ which behaves as $z^2 + \mathcal{O}(z)$ as $z\to \infty$. Define
\begin{equation}  \label{eq:NN}
 N_1(z) = r_1\frac{z+1}{r(z)}, \qquad  N_2(z) = \frac{z^2-1}{r(z)}, \qquad N_3(z) = r_3\frac{z-1}{r(z)},
\end{equation}
where $r_1,r_3$ are constants selected so that $N_1(1) = 1 = N_3(-1)$. Taking $N_0$ as in \eqref{eq:N0} relations \eqref{jump1}-\eqref{jump2} can be verified directly. Moreover, $N_0(z) =  \mathbb{I} + \mathcal{O}(1/z), z \to \infty$. From \eqref{jump1}-\eqref{jump2}, it follows that $\det N_0(z)$ is analytic in ${\mathbb{C}} \setminus \{-\mu,-1,1,\lambda\}$. The behavior of $\det N_0(z)$ in a neighborhood of any one of theses extreme points, say $\zeta$, is at worst like $\mathcal{O}(|z-\zeta|^{-1/2}), z \to \zeta$, so these singularities are removable and $\det N_0(z)$ is an entire function and because of its behavior at $\infty$ it is constantly equal to $1$.

\section{Parametrices around the endpoints}  \label{sec7}

Unfortunately, the jumps for $S$ on the lips $\Sigma_\pm^{a,b}$ and $\Sigma_\pm^{c,d}$ or $\Sigma_\pm^{c^*,d}$
do not tend uniformly to the identity matrix.
The uniformity is violated near the endpoints $a,b,c$ or $c^*,d$ of the intervals. We need to make a local analysis near each of these endpoints and
construct a parametrix that describes the local behavior near such a point, and which matches the global parametrix outside a neighborhood.
We will do this only for the point $b$, but the analysis is similar for the other three points.

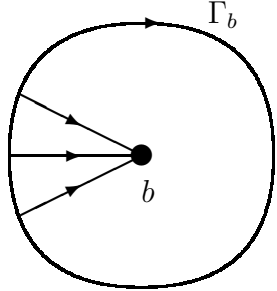
\begin{figure}[ht]
\unitlength=2.5pt
\begin{picture}(150,50)(10,25)
\thicklines
\put(100,50){\circle*{3}}
\put(100,70){\qbezier(0,0)(20,0)(20,-20)}
\put(120,50){\qbezier(0,0)(0,-20)(-20,-20)}
\put(100,30){\qbezier(0,0)(-20,0)(-20,20)}
\put(80,50){\qbezier(0,0)(0,20)(20,20)}
\put(100,50){\line(-1,0){20}}
\put(100,50){\line(-2,1){18.5}}
\put(100,50){\line(-2,-1){18.5}}
\put(100,43){$b$}
\put(88,50){\vector(1,0){3}}
\put(100,70){\vector(1,0){3}}
\put(88,56){\vector(2,-1){3}}
\put(88,44){\vector(2,1){3}}
\put(110,70){$\Gamma_b$}
\end{picture}
\caption{Parametrix around $b$}
\label{fig:par0}
\end{figure}

The idea is to approximate the Riemann-Hilbert problem for $S$ inside a curve $\Gamma_b$ around $b$ (see Figure \ref{fig:par0}) by a model
Riemann-Hilbert problem for a matrix $P_b$ which matches the global parametrix $N$ on $\Gamma_b$ with an error $\O(1/n)$. Such a local parametrix was
constructed earlier for orthogonal polynomials on $[-1,1]$ with Jacobi type weights in \cite{KMVAV}, and this is for a $2\times 2$ Riemann-Hilbert problem.
Observe that all the jumps for $S$ inside $\Gamma_b$ are of the form
\[  \begin{pmatrix} J & 0 \\ 0 & 1 \end{pmatrix}, \]
where $J$ is a $2\times 2$ matrix and $0$ is a row/column vector containing zeros, so basically the Riemann-Hilbert problem for $S$ near $b$ behaves like a $2\times 2$ Riemann-Hilbert problem,
so that we can use the construction from \cite{KMVAV} with some modifications.

The $2 \times 2$ matrix $\Psi$ that was considered in \cite[\S 6, p. 365]{KMVAV} solves the following Riemann-Hilbert
problem on the system of contours $\Sigma_\Psi = \gamma_1 \cup \gamma_2 \cup \gamma_3$, with
\[   \gamma_1 = \{ re^{2\pi i/3}: r > 0 \}, \quad \gamma_2 = (-\infty,0], \quad \gamma_3 = \{re^{-2\pi i/3}: r >0 \}, \]
with the orientation toward the point $0$:
\begin{itemize}
  \item $\Psi$ is analytic in $\mathbb{C} \setminus \Sigma_\Psi$,
  \item $\Psi$ satisfies the jump conditions
   \[\Psi_+(\zeta) = \Psi_-(\zeta) \begin{pmatrix} 1 & 0 \\ e^{\beta\pi i} & 1 \end{pmatrix}, \qquad \zeta \in \gamma_1, \]
   \[\Psi_+(\zeta) = \Psi_-(\zeta) \begin{pmatrix} 0 & 1 \\ -1 & 0 \end{pmatrix}, \qquad \zeta \in \gamma_2, \]
   \[\Psi_+(\zeta) = \Psi_-(\zeta) \begin{pmatrix} 1 & 0 \\ e^{-\beta\pi i} & 1 \end{pmatrix}, \qquad \zeta \in \gamma_3, \]
  \item for $\beta < 0$
          \[  \Psi(\zeta) = \O \begin{pmatrix}  |\zeta|^{\beta/2} &  |\zeta|^{\beta/2} \\
                                         |\zeta|^{\beta/2} & |\zeta|^{\beta/2} \end{pmatrix}, \qquad \zeta \to 0, \]
       for $\beta = 0$
          \[  \Psi(\zeta) = \O \begin{pmatrix} \log |\zeta| &  \log |\zeta| \\
                                         \log |\zeta| & \log |\zeta| \end{pmatrix}, \qquad \zeta \to 0, \]
       and for $\beta >0$
          \[  \Psi(\zeta) = \O \begin{pmatrix}  |\zeta|^{\beta/2} &  |\zeta|^{-\beta/2} \\
                            |\zeta|^{\beta/2} & |\zeta|^{-\beta/2} \end{pmatrix}, \qquad \zeta \to 0 \textrm{ and } |\arg \zeta| < \frac{2\pi}{3}, \]
          \[  \Psi(\zeta) = \O \begin{pmatrix}  |\zeta|^{-\beta/2} &  |\zeta|^{-\beta/2} \\
                            |\zeta|^{-\beta/2} & |\zeta|^{-\beta/2} \end{pmatrix}, \qquad \zeta \to 0 \textrm{ and } \frac{2\pi}{3} < |\arg \zeta| <\pi. \]
\end{itemize}
The matrix $\Psi$ is explicitly given by \cite[Thm. 6.3]{KMVAV}
\[  \Psi(\zeta) = \begin{pmatrix}  I_\beta(2\zeta^{1/2}) & \frac{i}{\pi} K_\beta(2\zeta^{1/2}) \\
                                   2\pi i \zeta^{1/2} I_\beta'(2\zeta^{1/2}) & -2\zeta^{1/2} K_\beta'(2\zeta^{1/2})
                  \end{pmatrix}, \qquad |\arg \zeta| < \frac{2\pi}{3}, \]
\begin{multline*}
\Psi(\zeta) = \begin{pmatrix}  \frac12 H_\beta^{(1)}(2(-\zeta)^{1/2}) & \frac12 H_\beta^{(2)}(2(-\zeta)^{1/2}) \\
                                   \pi \zeta^{1/2} (H_\beta^{(1)})'(2(-\zeta)^{1/2}) & \pi\zeta^{1/2} (H_\beta^{(2)})'(2(-\zeta)^{1/2})
                  \end{pmatrix}  \begin{pmatrix} e^{\beta\pi i/2} & 0 \\ 0 & e^{-\beta\pi i/2} \end{pmatrix}, \\
\qquad \frac{2\pi}{3} < \arg \zeta < \pi,
\end{multline*}
\begin{multline*}
\Psi(\zeta) = \begin{pmatrix}  \frac12 H_\beta^{(2)}(2(-\zeta)^{1/2}) & -\frac12 H_\beta^{(1)}(2(-\zeta)^{1/2}) \\
                                   -\pi \zeta^{1/2} (H_\beta^{(2)})'(2(-\zeta)^{1/2}) & \pi\zeta^{1/2} (H_\beta^{(1)})'(2(-\zeta)^{1/2})
                  \end{pmatrix}  \begin{pmatrix} e^{-\beta\pi i/2} & 0 \\ 0 & e^{\beta\pi i/2} \end{pmatrix}, \\
\qquad -\pi < \arg \zeta < -\frac{2\pi}3,
\end{multline*}
where $I_\beta$ and $K_\beta$ are modified Bessel functions \cite[\S 10.25]{NIST} and $H_\beta^{(1)}$ and $H_\beta^{(2)}$ are Hankel functions
\cite[\S 10.2]{NIST}.
The asymptotic behavior of the modified Bessel functions $I_\beta$ and $K_\beta$ \cite[\S 10.40]{NIST} shows that
\begin{equation}   \label{eq:Pbasym}
  \Psi(\zeta) = \begin{pmatrix} \frac{1}{\sqrt{2\pi}\zeta^{1/4}} & 0 \\ 0 & \sqrt{2\pi} \zeta^{1/4} \end{pmatrix}
          \frac{1}{\sqrt{2}} \begin{pmatrix}  1 + \O(\zeta^{-1/2}) & i + \O(\zeta^{-1/2}) \\ i + \O(\zeta^{-1/2} & 1 + \O(\zeta^{-1/2}) \end{pmatrix}
     \begin{pmatrix}  e^{2\zeta^{1/2}} & 0 \\ 0 & e^{-2\zeta^{1/2}} \end{pmatrix}
\end{equation}
when $\zeta \to \infty$ in the sector $|\arg \zeta| < 2\pi/3$. The same asymptotic formula holds in the regions $2\pi/3 < |\arg \zeta| < \pi$
if one uses the asymptotic behavior of the Hankel functions $H_\beta^{(1)}$ and $H_\beta^{(2)}$ \cite[\S 10.17]{NIST}.

We use this $2\times 2$ matrix $\Psi$ (in fact we use $\Psi^{-T} = (\Psi^{-1})^T$, the transpose of the inverse of $\Psi$) to construct the parametrix $P_b$ around the point $b$ as follows. We define the contours $\Sigma_{\pm}^{a,b}$ inside $\Gamma_b$ as the preimages of the rays $\gamma_1, \gamma_3$
under the mapping $\zeta = (n+m)^2 g_1^2(z)/4$, where $g_1$ is given in \eqref{eq:g}. Then the parametrix $P_b$ is given by
\begin{equation}  \label{eq:Pb}
  P_b(z) = E_n^b \begin{pmatrix} \Psi^{-T}\Bigl(\frac{(n+m)^2g_1^2(z)}{4}\Bigr) & 0 \\[10pt] 0 & 1  \end{pmatrix} \begin{pmatrix} W_1(z) & 0 & 0 \\ 0 & 1/W_1(z) & 0 \\ 0 & 0 & 1 \end{pmatrix}
        \begin{pmatrix} \Phi_1^{\frac{n+m}2} & 0 & 0 \\  0 & \Phi_1^{-\frac{n+m}2} & 0 \\ 0 & 0 & 1 \end{pmatrix},
\end{equation}
with $W_1(z) = \bigl(2\pi i (z-a)^\alpha (z-b)^\beta\bigr)^{1/2}$ and
\[  E_n^b(z) = N(z) \begin{pmatrix} B_n  & 0 \\ 0 & 1 \end{pmatrix}, \]
with
\[      B_n = \begin{pmatrix} 1/W_1 & 0 \\ 0 & W_1 \end{pmatrix} \frac1{\sqrt{2}} \begin{pmatrix} 1 & i \\ i & 1 \end{pmatrix}
    \begin{pmatrix} \frac1{\sqrt{\pi (n+m)} g_1^{1/2}} & 0 \\ 0 & \sqrt{\pi (n+m)} g_1^{1/2} \end{pmatrix}.   \]
One can then verify, after some calculations and by using the asymptotic behavior \eqref{eq:Pbasym},
that $P_b N^{-1} = \mathbb{I} + \O\Bigl(\frac1{n+m}\Bigr)$ on the contour $\Gamma_b$.

\bigskip

The parametrix $P_a$ around $a$ can be constructed in a similar way and uses the Bessel functions of order $\alpha$.
For $c$ and $d$ we proceed in the same way when $\textup{supp}(\nu_2)=[c,d]$ and $c^*<c<d$ but observe that now the jumps of $S$ near $c$ and $d$ are of the form
\[   \begin{pmatrix} 1 & 0 \\ 0 & \hat{J} \end{pmatrix}, \]
with $\hat{J}$ a $2 \times 2$ matrix. So the parametrices $P_c$ and $P_d$ can also be constructed using the function $\Psi$ (but with $\gamma$ or $\delta$)
and $P_c$ contains the Bessel functions of order $\gamma$, whereas $P_d$ contains the Bessel functions of order $\delta$. In particular we have
\begin{equation}  \label{eq:Pd}
   P_d = E_n^d \begin{pmatrix} 1 & 0 \\[10pt] 0 &  \Psi^{-T}\bigl(\frac{m^2g_2^2(z)}{4}\bigr) \end{pmatrix}
        \begin{pmatrix} 1 & 0 & 0 \\ 0 & W_2(z) & 0 \\ 0 & 0 & 1/W_2(z)  \end{pmatrix}
        \begin{pmatrix} 1 & 0 & 0  \\ 0 & \Phi_2^{\frac{m}2} & 0 \\  0 & 0 & \Phi_2^{-\frac{m}2} \end{pmatrix},
\end{equation}
with $W_2(z) = \bigl( 2\pi i(z-c)^\gamma (z-d)^\delta \bigr)^{1/2}$ and
\[  E_n^d(z) = N(z) \begin{pmatrix} 1 & 0  \\ 0 & C_n \end{pmatrix},  \]
with
\[   C_n = \begin{pmatrix} 1/W_2 & 0 \\ 0 & W_2 \end{pmatrix} \frac1{\sqrt{2}} \begin{pmatrix} 1 & i \\ i & 1 \end{pmatrix}
    \begin{pmatrix} \frac1{\sqrt{\pi m} g_2^{1/2}} & 0 \\ 0 & \sqrt{\pi m} g_2^{1/2} \end{pmatrix}.   \]
Then $P_d N^{-1} = \mathbb{I} + \O\Bigl(\frac1{m}\Bigr)$ on the contour $\Gamma_d$.
When $\textup{supp}(\nu_2) = [c^*,d]$ with $c < c^* < d$ we need another parametrix $P_{c^*}$ around $c^*$. The point $c^*$ is a soft edge
and the density $\nu_2'$ vanishes near $c^*$ as $(x-c^*)^{1/2}$. Near $c^*$ we look for a local Riemann-Hilbert problem with contours as in
Figure \ref{fig:parc*}

\begin{figure}[ht]
\unitlength=2.5pt
\begin{picture}(150,50)(10,25)
\thicklines
\put(100,50){\circle*{3}}
\put(100,70){\qbezier(0,0)(20,0)(20,-20)}
\put(120,50){\qbezier(0,0)(0,-20)(-20,-20)}
\put(100,30){\qbezier(0,0)(-20,0)(-20,20)}
\put(80,50){\qbezier(0,0)(0,20)(20,20)}
\put(100,50){\line(-1,0){20}}
\put(100,50){\line(1,0){20}}
\put(100,50){\line(2,1){18.5}}
\put(100,50){\line(2,-1){18.5}}
\put(100,43){$c^*$}
\put(88,50){\vector(1,0){3}}
\put(112,50){\vector(1,0){1}}
\put(100,70){\vector(1,0){3}}
\put(112,56){\vector(2,1){1}}
\put(112,44){\vector(2,-1){1}}
\put(110,70){$\Gamma_{c^*}$}
\end{picture}
\caption{Parametrix around $c^*$}
\label{fig:parc*}
\end{figure}
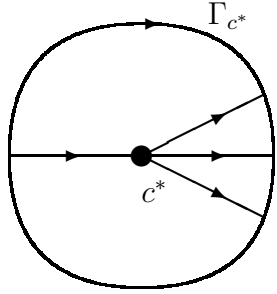

\noindent The jumps on these contours are all of the form
\[    \begin{pmatrix} 1 & 0 \\ 0 & \hat{J} \end{pmatrix}, \]
and so locally the problem reduces to a $2 \times 2$ problem. It is well known that around a soft edge one can use Airy functions
for the local parametrix, and the matching on the boundary $\Gamma_{c^*}$ can be achieved in a similar way as above, but by using the 
asymptotic behavior of the Airy function instead of Bessel functions. See, e.g., \cite[\S 7]{BleherKuijl} where this has been done in detail.
For $c=c^*$ there is a transition from soft edge $(c<c^*)$ to hard edge $(c>c^*)$. We will not deal with this special case since it requires a 
different parametrix in terms of Painlev\'e transcendents.

\section{Asymptotics for the type I multiple orthogonal polynomials}   \label{sec8}

The final transformation is
\begin{equation}  \label{eq:R}
    R(z) = \begin{cases}
          S(z) N^{-1}(z), & z \textrm{ outside } \Gamma_a, \Gamma_b, \Gamma_c, \Gamma_d, \\
          S(z) P_e^{-1}(z), & z \textrm{ inside } \Gamma_e,\ e \in \{a,b,c \textup{ or } c^*,d\}.
             \end{cases}
\end{equation}
The contours of the Riemann-Hilbert problem for $R$ are given in Figure \ref{fig:RHP-R} for the case when $\textup{supp}(\nu_2)=[c,d]$ (top picture).
When $\textup{supp}(\nu_2)=[c^*,d]$ (bottom picture) one replaces $c$ by $c^*$ and there is an extra horizontal contour from $c$ to $\Gamma_{c^*}$.

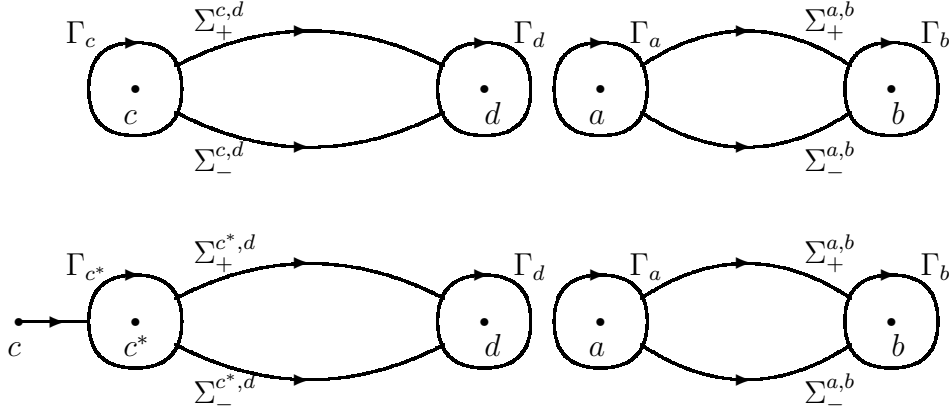
\begin{figure}[ht]
\unitlength=2.2pt
\begin{picture}(200,80)(-20,0)
\thicklines
\put(20,60){\circle*{1.5}}
\put(18,54){$c$}
\put(80,60){\circle*{1.5}}
\put(80,54){$d$}
\put(100,60){\circle*{1.5}}
\put(98,54){$a$}
\put(150,60){\circle*{1.5}}
\put(150,54){$b$}
\put(100,60){\qbezier(7,4)(25,16)(43,4)}
\put(100,60){\qbezier(7,-4)(25,-16)(43,-4)}
\put(20,60){\qbezier(7,4)(30,16)(53,4)}
\put(20,60){\qbezier(7,-4)(30,-16)(53,-4)}
\put(123,50){\vector(1,0){3}}
\put(123,70){\vector(1,0){3}}
\put(49,50){\vector(1,0){1}}
\put(49,70){\vector(1,0){1}}
\put(135,70){$\Sigma_+^{a,b}$}
\put(135,46){$\Sigma_-^{a,b}$}
\put(30,70){$\Sigma_+^{c,d}$}
\put(30,46){$\Sigma_-^{c,d}$}
\put(12,60){\qbezier(0,0)(0,8)(8,8)}
\put(20,68){\qbezier(0,0)(8,0)(8,-8)}
\put(28,60){\qbezier(0,0)(0,-8)(-8,-8)}
\put(20,52){\qbezier(0,0)(-8,0)(-8,8)}
\put(72,60){\qbezier(0,0)(0,8)(8,8)}
\put(80,68){\qbezier(0,0)(8,0)(8,-8)}
\put(88,60){\qbezier(0,0)(0,-8)(-8,-8)}
\put(80,52){\qbezier(0,0)(-8,0)(-8,8)}
\put(92,60){\qbezier(0,0)(0,8)(8,8)}
\put(100,68){\qbezier(0,0)(8,0)(8,-8)}
\put(108,60){\qbezier(0,0)(0,-8)(-8,-8)}
\put(100,52){\qbezier(0,0)(-8,0)(-8,8)}
\put(142,60){\qbezier(0,0)(0,8)(8,8)}
\put(150,68){\qbezier(0,0)(8,0)(8,-8)}
\put(158,60){\qbezier(0,0)(0,-8)(-8,-8)}
\put(150,52){\qbezier(0,0)(-8,0)(-8,8)}
\put(20,68){\vector(1,0){1}}
\put(80,68){\vector(1,0){1}}
\put(100,68){\vector(1,0){1}}
\put(150,68){\vector(1,0){1}}
\put(8,68){$\Gamma_c$}
\put(85,68){$\Gamma_d$}
\put(105,68){$\Gamma_a$}
\put(155,68){$\Gamma_b$}

\put(20,20){\circle*{1.5}}
\put(18,14){$c^*$}
\put(0,20){\circle*{1.5}}
\put(-2,14){$c$}
\put(80,20){\circle*{1.5}}
\put(80,14){$d$}
\put(100,20){\circle*{1.5}}
\put(98,14){$a$}
\put(150,20){\circle*{1.5}}
\put(150,14){$b$}
\put(0,20){\line(1,0){12}}
\put(100,20){\qbezier(7,4)(25,16)(43,4)}
\put(100,20){\qbezier(7,-4)(25,-16)(43,-4)}
\put(20,20){\qbezier(7,4)(30,16)(53,4)}
\put(20,20){\qbezier(7,-4)(30,-16)(53,-4)}
\put(123,10){\vector(1,0){3}}
\put(123,30){\vector(1,0){3}}
\put(49,10){\vector(1,0){1}}
\put(49,30){\vector(1,0){1}}
\put(135,30){$\Sigma_+^{a,b}$}
\put(135,6){$\Sigma_-^{a,b}$}
\put(30,30){$\Sigma_+^{c^*,d}$}
\put(30,6){$\Sigma_-^{c^*,d}$}
\put(12,20){\qbezier(0,0)(0,8)(8,8)}
\put(20,28){\qbezier(0,0)(8,0)(8,-8)}
\put(28,20){\qbezier(0,0)(0,-8)(-8,-8)}
\put(20,12){\qbezier(0,0)(-8,0)(-8,8)}
\put(72,20){\qbezier(0,0)(0,8)(8,8)}
\put(80,28){\qbezier(0,0)(8,0)(8,-8)}
\put(88,20){\qbezier(0,0)(0,-8)(-8,-8)}
\put(80,12){\qbezier(0,0)(-8,0)(-8,8)}
\put(92,20){\qbezier(0,0)(0,8)(8,8)}
\put(100,28){\qbezier(0,0)(8,0)(8,-8)}
\put(108,20){\qbezier(0,0)(0,-8)(-8,-8)}
\put(100,12){\qbezier(0,0)(-8,0)(-8,8)}
\put(142,20){\qbezier(0,0)(0,8)(8,8)}
\put(150,28){\qbezier(0,0)(8,0)(8,-8)}
\put(158,20){\qbezier(0,0)(0,-8)(-8,-8)}
\put(150,12){\qbezier(0,0)(-8,0)(-8,8)}
\put(20,28){\vector(1,0){1}}
\put(80,28){\vector(1,0){1}}
\put(100,28){\vector(1,0){1}}
\put(150,28){\vector(1,0){1}}
\put(7,20){\vector(1,0){1}}
\put(8,28){$\Gamma_{c^*}$}
\put(85,28){$\Gamma_d$}
\put(105,28){$\Gamma_a$}
\put(155,28){$\Gamma_b$}
\end{picture}
\caption{Riemann-Hilbert problem for the final matrix $R$}
\label{fig:RHP-R}
\end{figure}

The jumps of $S$ on $[a,b]$, $[c,d]$ or $[c^*,d]$ and the
lips of the lenses $\Sigma^{a,b}$ and $\Sigma^{c,d}$ or $\Sigma^{c^*,d}$ inside the curves $\Gamma_a, \Gamma_b, \Gamma_c$ or $\Gamma_{c^*}, \Gamma_d$ are eliminated by jumps that the global parametrix $N$ and each of the local parametrices $P_a,P_b,P_c$ or $P_{c^*},P_d$ have at those contours. On the remaining contours the jumps
tend to the identity matrix $\mathbb{I}$ uniformly, at the rate $\O(1/n)$ on $\Gamma_a,\Gamma_b,\Gamma_c$ or $\Gamma_{c^*},\Gamma_d$, and exponentially fast
at the remaining lips of the lenses $\Sigma^{a,b}$ and $\Sigma^{c,d}$ or $\Sigma^{c^*,d}$. The matrix $R$ for this Riemann-Hilbert problem then
converges uniformly on $\mathbb{C}$ to the identity matrix (see, e.g., \cite[\S7.5]{Deift} \cite[Thm. 3.1]{Arno})
\[  \lim_{n \to \infty} \| R - \mathbb{I} \|_\infty = 0, \]
and in fact the rate of convergence is of the same order as the rate at which the jumps converge to the identity matrix,
\begin{equation}  \label{eq:Rasymp}
    \| R(z) - \mathbb{I} \|_\infty = \O(1/n).
\end{equation}

\begin{theorem}
Let $A_{n,m}, B_{n,m}$ be the type I multiple orthogonal polynomials for a Nikishin system with measures $(\mu_1,\mu_2)$ on $[a,b]$
satisfying \eqref{eq:w}--\eqref{eq:mu1}, with a measure $\sigma$ on $[c,d]$ satisfying \eqref{eq:sigma}. Let $(n,m)$ be multi-indices that tend to infinity but for which $m/(n+m)=q_1$ remains constant, with $0 < q_1 \leq 1/2$. Then, uniformly
on compact subsets of $\mathbb{C} \setminus ([a,b] \cup [c,d])$
\begin{multline*}
   A_{n,m}(z) = \bigl[ N_1(\psi_1(z)) + \O(1/n) \bigr] \frac{D_0(\infty)}{D_1(z)} e^{(n+m)g_1(z)-mg_2(z)+(n+m)\ell_1} \\
              - \bigl[ N_1(\psi_2(z)) + \O(1/n) \bigr] \frac{D_0(\infty)}{D_2(z)} e^{mg_2(z)+(n+m)(\ell_1+\ell_2)}\int_c^d \frac{d\sigma(t)}{z-t},
\end{multline*}
and
\[   B_{n,m}(z) = \bigl[N_1(\psi_2(z)) + \O(1/n)\bigr] \frac{D_0(\infty)}{D_2(z)} e^{mg_2(z) + (n+m)(\ell_1+\ell_2)}, \]
where $g_1$ and $g_2$ are given in \eqref{eq:g}, $\ell_1$ and $\ell_2$ are given in \eqref{eq:Uab}--\eqref{eq:Ucd}, $N_1$ is
given in \eqref{eq:NN}, and $D_0, D_1, D_2$ are given in \eqref{eq:D}.
Furthermore
\[   A_{n,m}(z) + B_{n,m} \int_c^d \frac{d\sigma(t)}{z-t} =  \bigl[ N_1(\psi_1(z)) + \O(1/n) \bigr] \frac{D_0(\infty)}{D_1(z)} e^{(n+m)g_1(z)-mg_2(z)+(n+m)\ell_1}. \]
\end{theorem}

\begin{proof}
We need to undo all the transformations from the original matrix $X$ in Section \ref{sec1} to $R$ and then use the asymptotic behavior \eqref{eq:Rasymp}
for $R$. From \eqref{eq:U} we find
\[   A_{n,m}(z) = X_{1,2} = U_{1,2} - U_{1,3} \int_c^d \frac{d\sigma(t)}{z-t} ,  \]
and
\[   B_{n,m}(z) = X_{1,3} = U_{1,3} .  \]
From \eqref{eq:V} we find
\begin{eqnarray*}
   U_{1,2} &=& V_{1,2} e^{(n+m)\ell_1+(n+m)g_1-mg_2}, \\
   U_{1,3} &=& V_{1,3} e^{(n+m)(\ell_1+\ell_2)+mg_2}.
\end{eqnarray*}
Since $z$ is on a compact subset of $\mathbb{C}\setminus ([a,b] \cup [c,d])$, we will take the lenses around $[a,b]$ and $[c,d]$ and the
neighborhoods around $a,b,c,d$ sufficiently small so that the compact subset is outside the system of curves in Figure \ref{fig:RHP-R}. Then
$V = S$ for the matrix $S$ in Section \ref{sec5}. Finally, from \eqref{eq:R} we find that
\begin{eqnarray}
   S_{1,2} &=& R_{1,1}N_{1,2} + R_{1,2}N_{2,2} + R_{1,3} N_{3,2} \label{eq:S12}\\
   S_{1,3} &=& R_{1,1}N_{1,3} + R_{1,2}N_{2,3} + R_{1,3} N_{3,3} .  \label{eq:S13}
\end{eqnarray}
Then by using the asymptotic behavior in \eqref{eq:Rasymp} and the expressions for $N$ from \eqref{eq:N} and \eqref{eq:N0}, we find the required asymptotic result away from $[a,b] \cup [c,d]$.

When $\textup{supp}(\nu_2)=[c^*,d]$ with $c < c^* < d$, then the asymptotic formula for $A_{n,m}$ still holds on $\mathbb{C} \setminus ([a,b] \cup [c,d])$ 
and it is not valid on $[c,c^*]$ since $A_{n,m}(z)$ contains the function $w(z)$ from \eqref{eq:w} and this function makes a jump
over the interval $[c,c^*]$. The asymptotic formula for $B_{n,m}$ is true on $[c,c^*-\epsilon]$, taking into account that $g_2^{\pm}(z) = -U(x;\nu_2) \pm i\pi$ there. Note however that the functions $N_1, D_1, D_2$ are different from the case where $\textup{supp}(\nu_2) = [c,d]$ because the Riemann
surface is different.
\end{proof}

\bigskip
Notice that using \eqref{inverse}, we can rewrite \eqref{orthAB}  as
\[
   \int_a^b  \Bigl( A_{n,m}(x) \tilde{w}(x) + \tilde{B}_{n,m}(x) \Bigr) x^k w(x)w_1(x)\, dx = 0, \qquad 0 \leq k \leq n+m-2,
\]
where $\tilde{B}_{n,m} = -B_{n,m} -\ell A_{n,m}$ and $\deg(\tilde{B}_{n,m}) \leq m-1$ when $n < m$. From here, reasoning as in Section \ref{sec3}, relations \eqref{orto1}-\eqref{AB/H} may be replaced by
\[
 \int_a^b  H_{n,m}(x) \frac{ A_{n,m}(x)\tilde{w}(x) + \tilde{B}_{n,m}(x) }{H_{n,m}(x)} x^{k} w(x)w_1(x)\, dx = 0,\qquad k=0,\ldots,n+m-2.
\]

\[
   \int_c^d A_{n,m}(x) x^k  \frac{d\tilde{\sigma}(x)}{H_{n,m}(x)} = 0, \qquad 0 \leq k \leq n-2,
\]
and
\[
   \frac{A_{n,m}(x)\tilde{w}(x) + \tilde{B}_{n,m}(x)}{H_{n,m}(x)} = \int_c^d \frac{\tilde{B}_{n,m}(t)}{x-t} \frac{d\tilde{\sigma}(t)}{H_{n,m}(t)}, \qquad x \notin [c,d],
\]
where $H_{n,m}$ represents the same polynomial as before.

Now, if we replace in  \eqref{eq:Uab}-\eqref{eq:Ucd} the constant $q_1$ by $\tilde{q}_1 = \lim_{n} \frac{n}{n+m} = 1-{q}_1, 1/2 < q_1 < 1$ and $(\tilde{\nu}_1,\tilde{\nu}_2)$ is the solution of the corresponding variational relations, then $\tilde{\nu}_1$ gives the normalized asymptotic zero distribution of the zeros of the polynomials $H_{n,m}$ (as before) but $\tilde{\nu}_2$ gives the normalized asymptotic zero distribution of the zeros of the polynomials $A_{n,m}$. Repeating Sections 4--7 one obtains an analogue of Theorem 1 for the case when $1/2 < q_1 < 1$. The details are left to the reader.

The asymptotic behavior on the intervals $[a,b]$ and $[c,d]$ can be obtained in a similar way. The only difference is that we need to use the relation
between $S$ and $V$ inside the lenses, and $S \neq V$ there. The asymptotic behavior of $B_{n,m}$ on $(c,d)$ is then given by the following theorem.

\begin{theorem}
Let $B_{n,m}$ be the type I multiple orthogonal polynomials for a Nikishin system with measures $(\mu_1,\mu_2)$ on $[a,b]$
satisfying \eqref{eq:w}--\eqref{eq:mu1}, with a measure $\sigma$ on $[c,d]$ satisfying \eqref{eq:sigma}. Let $(n,m)$ be multi-indices that tend to infinity but for which $m/(n+m)=q_1$ remains constant, with $0 < q_1 \leq 1/2$. Then for $\textup{supp}(\nu_2)=[c,d]$ one has uniformly
on closed subintervals of $(c,d)$
\[   B_{n,m}(x) = -2[N_1(\psi_1^+(x)) + \mathcal{O}(1/n)] \frac{D_0(\infty)}{|D_2^+(x)|} e^{-mU(x;\nu_2)} \cos \Bigl( m\pi \varphi_2(x) - \arg D_2^+(x) \Bigr). \]
If $\textup{supp}(\nu_2) = [c^*,d]$ then this asymptotic formula holds unformly on closed intervals of $(c^*,d)$.
\end{theorem}

\begin{proof}
Since $x$ is now on a closed subinterval of $(c,d)$, we need to use $S$ inside the lens around $(c,d)$. We will use the limiting values $S_+$
to get the behavior of $B_{n,m}$ on $(c,d)$. The relation between $S$ and $V$, as described in Section \ref{sec5}, is
\[   V_{1,3} = S_{1,3} - S_{1,2} \frac{\Phi_2^{-m}}{v_2} . \]
We avoid the points $c$ and $d$ by taking the neighborhoods around those points small enough. Then $S$ and $R$ are related by \eqref{eq:S12}--\eqref{eq:S13}. The asymptotic behavior of $R$ in \eqref{eq:Rasymp} then gives
\[    B_{n,m}(x) = \left( [N_1(\psi_2^+) + \mathcal{O}(1/n)] \frac{D_0(\infty)}{D_2^+(x)}
 - [N_1(\psi_1^+)+\mathcal{O}(1/n)] \frac{D_0(\infty)}{D_1^+(x)v_2(x)} \right) e^{(n+m)(\ell_1+\ell_2)+mg_2^+}.  \]
Now, recall that on $(c,d)$ we have $v_2(x)D_1^+(x) = D_2^-(x)$, $\psi_2^+(x)=\psi_1^-(x)$, $N_1(\psi_1^+)=-N_1(\psi_1^-)$, $\Phi_2^+(x)= \exp(2\pi i \varphi_2(x))$, and
$g_2^+(x) = -U(x;\nu_2) + i\pi \varphi_2(x)$, see \eqref{eq:6.2}, Figure \ref{fig:Rimage}, and \eqref{eq:g2}. Combining all these relations then gives the required result. If $\textup{supp}(\nu_2) = [c^*,d]$ we need to use $S$ inside the lens around $(c^*,d)$. On $[c,c^*-\epsilon]$ the $B_{n,m}$ has
exponential behavior, see our remark at the end of the proof of previous theorem. 
\end{proof}

On the interval $(a,b)$ we have the following asymptotic result:

\begin{theorem}
Let $A_{n,m}, B_{n,m}$ be the type I multiple orthogonal polynomials for a Nikishin system with measures $(\mu_1,\mu_2)$ on $[a,b]$
satisfying \eqref{eq:w}--\eqref{eq:mu1}, with a measure $\sigma$ on $[c,d]$ satisfying \eqref{eq:sigma}. Let $(n,m)$ be multi-indices that tend to infinity but for which $m/(n+m)=q_1$ remains constant, with $0 < q_1 \leq 1/2$. Then, uniformly
on closed subintervals of $(a,b)$
\begin{multline*}  A_{n,m}(x) + B_{n,m}(x) \int_c^d \frac{d\sigma(t)}{x-t}  \\
 = 2 [N_1(\psi_1^+) + \mathcal{O}(1/n)] \frac{D_0(\infty)}{|D_1^+(x)|} e^{(n+m) U(x;\nu_1)} \cos \Bigl( (n+m)\varphi_1(x)-\arg D_1^+(x) \Bigr).
\end{multline*}
\end{theorem}

\begin{proof}
It follows from \eqref{eq:U} that
\[  U_{1,2} = A_{n,m}(x) + B_{n,m}(x) \int_c^d \frac{d\sigma(t)}{x-t}, \]
hence we need to get the asymptotic behavior of $U_{1,2}$. We will investigate this inside the lens around $[a,b]$ and away from the endpoints $a,b$
and only investigate the limiting values from above. The transformations \eqref{eq:V} and the relation between $S$ and $V$ show that
\[   U_{1,2} = \left( S_{1,2} - S_{1,1} \frac{\Phi_1^{-(n+m)}}{v_1}\right) e^{(n+m)\ell_1+(n+m)g_1-mg_2}. \]
Since $S=RN$ we then can use the asymptotic behavior \eqref{eq:Rasymp} to find
\begin{multline*}
  U_{1,2} = \left( [N_1(\psi_1^+) + \mathcal{O}(1/n)] \frac{D_0(\infty)}{D_1^+(x)} - [N_1(\psi_0^+)
+ \mathcal{O}(1/n)] \frac{D_0(\infty)}{D_0^+(x)v_1(x)} (\Phi_1^+)^{-(n+m)} \right) \\
  \times e^{(n+m)\ell_1+(n+m)g_1^+-mg_2} .
\end{multline*}
On $(a,b)$ one has by \eqref{eq:6.1} that $D_0^+(x)v_1(x)=D_1^-(x)$, and from Figure \ref{fig:Rimage} we see that $\psi_0^+(x)=\psi_1^-(x)$ and
$N_1(\psi_1^+)=-N_1(\psi_1^-)$. Furthermore,
$\Phi_1^+(x)=\exp(2\pi i \varphi_1(x))$, $g_1^+(x)=-U(x;\nu_1)+i\pi \varphi_1(x)$ and $g_1(x) = -U(x;\nu_2)$. Combining all this and using
the variational relation \eqref{eq:Uab} then gives the required result.
\end{proof}

One can also obtain the asymptotic behavior of $B_{n,m}$ around the endpoints $c$ and $d$ by using that $S = RP_{c}$ or $S = RP_{d}$ and then use
the parametrix $P_c$ or $P_d$ given in \eqref{eq:Pd}. This will give asymptotics in terms of Bessel functions $J_\gamma$ or $J_\delta$. 
When $\textup{supp}(\nu_2)=[c^*,d]$ the asymptotic behavior near $c^*$ will be in terms of the Airy function.
In a similar way one can also get the asymptotic behavior of the function $A_{n,m}+B_{n,m}w$ around the endpoints $a$ and $b$ by using the parametrices $P_a$ and $P_b$
in \eqref{eq:Pb}, resulting in a formula involving Bessel functions $J_\alpha$ or $J_\beta$. We do not give the resulting formulas but leave this to
the reader who is willing to do the necessary calculations.

\section{Asymptotics for the type II multiple orthogonal polynomial}   \label{sec9}

So far we only considered the type I multiple orthogonal polynomials $A_{n,m}, B_{n,m}$. However, one can also obtain the asymptotic behavior
of the type II multiple orthogonal polynomials $P_{n,m}$ because there is a simple relation between the Riemann-Hilbert problem for type I and
type II, see \cite[Thm. 4.1]{WVAGerKuijl} or \cite[Thm. 23.8.3]{Ismail},
\[  X^{-T} = \begin{pmatrix}
  P_{n,m}(z) & \displaystyle \int_a^b \frac{P_{n,m}(t)w_1(t)}{t-z}\, dt & \displaystyle \int_a^b \frac{P_{n,m}(t)w(t) w_1(t)}{t-z}\, dt \\
  -\gamma_1 P_{n-1,m}(z) & \displaystyle -\gamma_1 \int_a^b \frac{P_{n-1,m}(t)w_1(t)}{t-z}\, dt &
   \displaystyle -\gamma_1 \int_a^b \frac{P_{n-1,m}(t)w(t)w_1(t)}{t-z}\, dt \\
  -\gamma_2 P_{n,m-1}(z) & \displaystyle -\gamma_2 \int_a^b \frac{P_{n,m-1}(t)w_1(t)}{t-z}\, dt &
   \displaystyle -\gamma_2 \int_a^b \frac{P_{n,m-1}(t)w(t)w_1(t)}{t-z}\, dt
    \end{pmatrix}, \]
where
\[   \frac{1}{\gamma_{1}} = \int_a^b t^{n-1} P_{n-1,m}(t)w_1(t)\, dt, \quad
     \frac{1}{\gamma_{2}} = \int_a^b t^{m-1} P_{n,m-1}(t)w(t)w_1(t)\, dt.   \]
So in order to find the asymptotic behavior of $P_{n,m}(z)$, we need to investigate $X^{-T} = (X^{-1})^T$, i.e., the transpose of the inverse of $X$.
Note that
\[  P_{n,m}(z) = (X^{-T})_{1,1} = (U^{-T})_{1,1} = (V^{-T})_{1,1} e^{(n+m)g_1(z)}, \]
so that we only need to investigate the $(1,1)$-entry of $V^{-T}$. This gives

\begin{theorem}
Let $P_{n,m}$ be the type II multiple orthogonal polynomials for a Nikishin system with measures $(\mu_1,\mu_2)$ on $[a,b]$
satisfying \eqref{eq:w}--\eqref{eq:mu1}, with a measure $\sigma$ on $[c,d]$ satisfying \eqref{eq:sigma}. Let $(n,m)$ be multi-indices that tend to infinity but for which $m/(n+m)=q_1$ remains constant, with $0 < q_1 \leq 1/2$. Then, uniformly
on compact subsets of $\mathbb{C} \setminus [a,b]$
\begin{equation}   \label{eq:Pz}
   P_{n,m}(z) = \frac{D_0(z)}{D_0(\infty)} \bigl[ N_1(\psi_0(z)) + \mathcal{O}(1/n) \bigr] e^{(n+m)g_1(z)} ,
\end{equation}
where $g_1$ is given in \eqref{eq:g}. For $x$ on closed subintervals of $(a,b)$ one has
\begin{equation}  \label{eq:Px}
   P_{n,m}(x) = 2i \frac{|D_0^+(x)|}{D_0(\infty)} [N_1(\psi_0^+(x))+\mathcal{O}(1/n)]  \sin \Bigl( (n+m)\varphi_1 + \arg D_0^+(x) \Bigr) .
\end{equation}
\end{theorem}

\begin{proof}
Since we are on a compact subset of $\mathbb{C} \setminus [a,b]$, we only need to investigate $V$ outside the lens around $[a,b]$ and
the neighborhoods around $a$ and $b$. There we have that $S=V$ and $S = RN$, so that
\[  P_{n,m}(z) = [(RN)^{-T}]_{1,1} e^{(n+m)g_1(z)}.  \]
The asymptotic behavior of $R$ in \eqref{eq:Rasymp} gives $RN = N + \mathcal{O}(1/n)$, hence
\[      [(RN)^{-T}]_{1,1} = (N^{-T})_{1,1} + \mathcal{O}(1/n). \]
Use \eqref{eq:N} to write $N^{-T}$ in terms of $N_0^{-T}$ and
observe that $N_0$ and $N_0^{-T}$ obey the same Riemann-Hilbert problem, since the jumps $J$ of $N_0$ satisfy $J^{-T} = J$ and $N_0$ tends to the
identity matrix as $z \to \infty$. This gives the required asymptotic formula \eqref{eq:Pz}.

For $x$ on a closed subinterval of $(a,b)$ we need to use the behavior of $V$ inside the lens around $[a,b]$ but away from the endpoints $a$ and $b$.
We will use the limit from the upper half plane. There one has
\[   V^{-T} = S^{-T} \begin{pmatrix} 1 & 0 & 0 \\ \Phi_1^{-(n+m)}/v_1 & 1 & 0 \\ 0 & 0 & 1 \end{pmatrix}, \]
so that
\[  (V^{-T})_{1,1} = (S^{-T})_{1,1} + (S^{-T})_{1,2} \frac{\Phi_1^{-(n+m)}}{v_1}.  \]
Furthermore, we have $S = RN$, and the asymptotic behavior of $R$ gives $S = N + \mathcal{O}(1/n)$. From
$N_0^{-T} = N_0$ we then find
\[   P_{n,m}(x) = \left( [N_1(\psi_0^+) + \mathcal{O}(1/n)]\frac{D_0^+(x)}{D_0(\infty)}
+ [N_1(\psi_1^+) + \mathcal{O}(1/n)] \frac{D_1^+(x)}{D_0(\infty)}\frac{(\Phi_1^+)^{-(n+m)}}{v_1(x)} \right)
    e^{(n+m)g_1^+(x)} . \]
On $(a,b)$ one has $D_1^+(x) = v_1(x) D_0^-(x)$, see \eqref{eq:6.1}, $g_1^+(x) = -U(x;\nu_1) + i\pi \varphi_1(x)$, see \eqref{eq:g1}, and $\Psi_1^+(x) = \exp(2\pi i \varphi_1(x))$. Furthermore,
$\psi_0^+(x)=\psi_1^-(x)$ and $N_1(\psi_1^+)=-N_1(\psi_1^-)$ so that the required formula \eqref{eq:Px} follows.
\end{proof}

Note that this asymptotic formula does not contain the constants $\ell_1$ or $\ell_2$. This is because $P_{n,m}(z)$ is a monic polynomial.

\section{Concluding remarks}
In this paper we have used the Riemann-Hilbert problem and the Deift-Zhou steepest descent method for oscillatory Riemann-Hilbert problems
to obtain the asymptotics of the type I and type II multiple orthogonal polynomials for a Nikishin system of order two.
This Riemann-Hilbert problem uses $3\times 3$ matrix functions and we showed that many steps in the Riemann-Hilbert problem can be reduced
to a $2\times 2$ Riemann-Hilbert problem when the two intervals $[a,b]$ and $[c,d]$ are disjoint and not touching. The only steps where the
$3\times 3$ character of the problem is important is when we normalize the problem in Section \ref{sec4} using the solution of the vector equilibrium problem with Nikishin interaction from Section \ref{sec3}, and the construction of the global parametrix in Section \ref{sec6}. 
If the intervals $[a,b]$ and $[c,d]$ are touching, then the construction of the parametrix around the common point $a=d$ also requires
a local $3\times 3$ Riemann-Hilbert problem, but it is not clear what such a parametrix should contain. We believe it will be somewhat like the
local parametrix which was used in \cite{KlaasArno} around the common point of the two intervals in an Angelesco system, or the parametrix used
in \cite{AptWVAYatt} for the critical case $a=1/\sqrt{2}$ in that paper, but it will not be quite the 
same parametrix because in an Angelesco system the two intervals are repelling, whereas in a Nikishin system the two intervals are attracting.
Also the critical case $c=c^*$ is not considered in this paper. We believe the parametrix around the endpoint $c$ will be in terms of Painlev\'e
XXXIV, as was the case for the Angelesco case \cite{Yatt} and similar situations in random matrix theory \cite{ItsKuijlOst} and asymptotics for orthogonal polynomials \cite[\S 7.2]{WVA}. This is rather technical, so we decided not to deal with it in this paper.

\section*{Acknowledgements}
GLL was supported by Grant MTM 2015-65888-C4-2-P of Ministerio de Econom\'\i a y Competitividad, Spain, and WVA
was supported by KU Leuven research grant OT/12/073 and FWO research projects G.0864.13 and G.0864.16N. This work was initiated during a visit 
of WVA to Universidad Carlos III de Madrid and he is grateful for the support of the Departamento de Matem\'aticas of UC3M.
We are grateful to the referees for giving extra useful references and for pointing out that a soft edge $c^*$ is possible for the support of $\nu_2$.

\bigskip

\parbox{3in}{
Guillermo L\'opez Lagomasino \\
Departamento de Matem\'aticas \\
Universidad Carlos III de Madrid \\
Avenida de la Universidad 30 \\
ES-28911 Legan\'es (Madrid) \\
SPAIN \\
lago@math.uc3m.es}
\parbox{3in}{
Walter Van Assche \\
Department of Mathematics  \\
KU Leuven \\
Celestijnenlaan 200B box 2400 \\ 
BE-3001 Leuven \\
BELGIUM \\
walter@wis.kuleuven.be}

\end{document}